\RequirePackage{etoolbox}
\csdef{input@path}{{style/}{graphics/}}
\documentclass[MSNbibl,number,citesort,dvips]{arxbj}
\aid{0}
\volume{21}
\issue{1}
\pubyear{2015}
\firstpage{618}
\lastpage{646}
\doi{10.3150/13-BEJ582} %kopijuoti is PTS

\makeatletter

\newtheorem{theorem}{Theorem}
\newtheorem{lemma}{Lemma}
\newproclaim{definition}{Definition}
\newproclaim{Assumptions}{Assumptions}
\newremark{Remarks}{Remarks}
\newremark{Remark}{Remark}
\newproclaim{Property}{Property}

\def\comment#1{}
\renewcommand\vec[1]{\boldsymbol{#1}}
\newcommand{\E}{\mathbb{E}}
\newcommand{\Nat}{\mathbb{N}}
\newcommand{\indicator}{\mathbb{I}}
\newcommand{\real}{\mathbb{R}}
\renewcommand{\SS}{\mathcal{S}}

\renewcommand{\d}{p}
\newcommand{\Pcal}{\mathcal{P}}
\newcommand{\Qcal}{\mathcal{Q}}
\newcommand{\Ocal}{\mathcal{O}}
\newcommand{\Gcal}{\mathcal{G}}
\newcommand{\Hcal}{\mathcal{H}}
\newcommand{\hatG}{\hat{G}}
\newcommand{\Diam}{\operatorname{Diam}}
\newcommand{\supp}{\operatorname{supp}}

\newcommand{\aff}{\operatorname{aff}}
\newcommand{\conv}{\operatorname{conv}}
\newcommand{\interior}{\operatorname{relint}}
\newcommand{\extr}{\operatorname{extr}}
\newcommand{\vectheta}{\vec{\theta}}
\newcommand{\veceta}{{\vec{\eta}}}
\newcommand{\vecbeta}{{\vec{\beta}}}
\newcommand{\bd}{\operatorname{bd}}
\newcommand{\vol}{\operatorname{vol}}
\newcommand{\W}{W}

\newcommand{\hatveceta}{\hat{\veceta}}

\renewcommand{\AA}{A1}
\newcommand{\eqref}[1]{(\ref{#1})}
\newcommand{\Mult}{\operatorname{Mult}}
\renewcommand{\epsilon}{\varepsilon}
\renewcommand{\emptyset}{\varnothing}
\def\sfrac#1#2{#1/#2}

\def\afrac#1#2{#1/(#2)}
\def\vafrac#1#2{(#1)/(#2)}

\makeatother

\begin{document}
\begin{frontmatter}

\title{Posterior contraction of the population polytope in finite
admixture models}
\runtitle{Convergence of population polytope}

\begin{aug}
%%%% inicialai - be tarpu
\author{\inits{X.}\fnms{XuanLong} \snm{Nguyen}\ead
[label=e1]{xuanlong@umich.edu}}% \and
%%\runauthor{} %% auto
\address[]{Department of Statistics,
University of Michigan,
456 West Hall,
Ann Arbor, MI 48109-1107,
USA.\\ \printead{e1}}
\end{aug}

% HISTORY:
\received{\smonth{12} \syear{2012}}
\revised{\smonth{9} \syear{2013}}

% ABSTRACT
%
\begin{abstract}
We study the posterior contraction behavior of the
latent population structure that arises in admixture models
as the amount of data increases. We adopt the geometric view of
admixture models -- alternatively known as topic models --
as a data generating mechanism for points randomly
sampled from the interior of a (convex) population polytope,
whose extreme points
correspond to the population structure variables of interest.
Rates of posterior contraction
are established with respect to Hausdorff
metric and a minimum matching Euclidean metric defined on
polytopes. Tools developed include posterior asymptotics of
hierarchical models
and arguments from convex geometry.
\end{abstract}

% KEYWORDS
% visi is mazosios raides ir pagal abecele
%

\end{frontmatter}

%s1 #&#
\section{Introduction}\label{intro}
We study a class of hierarchical mixture
models for categorical data known as the admixtures, which were independently
developed in the landmark papers by
Pritchard, Stephens and Donnelly \cite{Pritchard-etal-00}
and Blei, Ng and Jordan \cite{Blei-etal-03}.
The former set of authors applied their modeling to population genetics,
while the latter considered applications in text processing and computer
vision, where their
models are more widely known as the \emph{latent Dirichlet allocation}
model, or a topic model. Admixture modeling has been applied to
and extended in a vast number of fields of engineering and sciences --
in fact, the Google scholar pages
for these two original papers alone combine for more than a dozen
thousands of citations.
In spite of their wide uses,
asymptotic behavior of hierarchical models such as the admixtures
remains largely unexplored, to the best of our knowledge.

A finite admixture model posits that there are $k$ populations, each of which
is characterized by a $\Delta^d$-valued vector $\vectheta_j$
of frequencies for generating a set of discrete values $\{0,1,\ldots
,d\}$,
for $j=1,\ldots,k$. Here, $\Delta^d$ is the $d$-dimensional
probability simplex.
A sampled individual
may have mixed ancestry and as a result inherits some fraction of
its values from each of its ancestral populations.
Thus, an individual is associated with a proportion vector
$\vecbeta= (\beta_1,\ldots,\beta_k)
\in\Delta^{k-1}$, where $\beta_j$ denotes the proportion of the individual's
data that are generated according to population $j$'s frequency vector
$\vectheta_j$. This yields a vector of frequencies
$\veceta= \sum_{j=1}^{k}\beta_j \vectheta_j \in\Delta^d$ associated
with that individual.
In most applications, one does not observe $\veceta$ directly, but rather
an i.i.d. sample generated from a multinomial distribution parameterized
by $\veceta$. The collection of
$\vectheta_1,\ldots,\vectheta_k$ is referred to as the \emph
{population structure}
in the admixture.
In population genetics modeling, $\vectheta_j$ represents the allele
frequencies at each locus in an individual's genome from the $j$th population.
In text document modeling, $\vectheta_j$ represents the frequencies
of words generated by the $j$th topic, while an individual is a
document, that is, a collection of words. In computer vision,
$\vectheta_j$ represents the frequencies of objects generated by
the $j$th scenary topic, while an individual is a natural image,
that is, a collection of scenary objects.
The primary interest is
the inference of the population structure on the basis of sampled data.
In a Bayesian estimation setting, the population structure is assumed
random and endowed with a prior distribution --
accordingly one is interested in the behavior of the posterior
distribution of the population structure given the available data.

The goal of this paper is to obtain contraction rates of the
posterior distribution of the latent population
structure that arises in admixture models, as the amount of data
increases. Admixture models
present a canonical mixture model for categorical data
in which the population structure provides the support for the mixing measure.
Existing works on convergence behavior of
mixing measures in a mixture model are quite rare, in either
frequentist or Bayesian estimation literature.
Chen provided the optimal convergence rate of mixing
measures in several finite mixtures for univariate data \cite{Chen-95}
(see also \cite{Ishwaran-James-Sun-01}). Recent progress on
multivariate mixture
models include papers by Rousseau and Mengersen \cite{Rousseau-Mengersen-11}
and Nguyen \cite{Nguyen-11}. In \cite{Nguyen-11},
posterior contraction rates of mixing measures in several finite and
infinite mixture models for multivariate and continuous data were obtained.
Toussile and Gassiat established consistency of a penalized
MLE procedure for a finite admixture model \cite{Toussile-Gassiat}.
This issue has also attracted increased attention in machine learning.
Recent papers by Arora \textit{et al.} \cite{Arora-LDA-12} and
Anandkumar \textit{et al.} \cite{Anandkumar-LDA-12}
study convergence properties of certain computationally efficient
learning algorithms based on matrix factorization
techniques.
%under distributional assumptions so that such
%algorithms can be guaranteed to recover the population structure.

There are a number of questions that arise in the convergence
analysis of admixture models for categorical data. The first question is
to find a suitable metric in order to establish rates of convergence.
It would be ideal to establish convergence for each individual element
$\vectheta_i$, for $i=1,\ldots, k$. This is a challenging task due to
the problems of identifiability. A (relatively) minor issue
is known as ``label-switching'' problem. That is, one can identify
the collection of $\vectheta_i$'s only up to a permutation. A deeper
problem is that any $\vectheta_j$ that can be expressed as a convex combination
of the others $\vectheta_{j'}$ for $j' \neq j$ may be difficult to identify,
estimate, and analyze.
To get around this difficulty, we propose to study the
convergence of population structure variables through its convex hull
$G = \conv(\vectheta_1,\ldots,\vectheta_k)$, which shall be
referred to as the \emph{population polytope}.
Convergence of convex polytopes can be evaluated in terms of
Hausdorff metric $d_{\mathcal{H}}$, a metric commonly utilized in
convex geometry \cite{Schneider-93}. Moreover, under
some geometric identifiability conditions, it can be shown that
convergence in Hausdorff metric
entails convergence of all extreme points of the polytope
via a minimum-matching distance metric (defined in Section~\ref{Sec-statement}).
This is the theory we aim for in this paper.
Note however that in a typical setting of topic modeling
where $d \geq k$, all population structure variables $\vectheta
_1,\ldots,
\vectheta_k$ in general positions in $\Delta^d$
\emph{are} extreme points of the population polytope.
Thus, in this setting, convergence in Hausdorff metric entails
convergence of the population structure variables up to a
permutation of labels. Convergence behavior of (the posterior of)
non-extreme points among $\vectheta_1,\ldots, \vectheta_k$,
when $k > d$, remains elusive as of this writing.

%We
%posterior contraction rates of the population polytope. Hausdorff is
%a natural choice for analyzing estimators of sets (e.g.,

The second question in an asymptotic study of
a hierarchical model is how to address
multiple quantities that define the amount of empirical data.
The admixture model we consider has
two asymptotic quantities that play asymmetric roles --
$m$ is the number of individuals, and $n$ is the number of data points
associated
with each individual. Both $m$ and $n$ are allowed to increase to infinity.
A simple way to think about this asymptotic setting is to let $m$ go to
infinity,
while $n:= n(m)$ tends to infinity at a certain rate which may be constrained
with respect to $m$. Let $\Pi$ be a prior distribution on
variables $\vectheta_1,\ldots, \vectheta_k$.
The goal is to derive a vanishing sequence of $\delta_{m,n}$,
depending on both $m$ and $n$,
such that the posterior distribution of the $\vectheta_i$'s satisfies,
for some sufficiently large constant $C$,
\[
\Pi \bigl(d_{\mathcal{H}}(G,G_0) \geq C \delta_{m,n} |
\SS _{[n]}^{[m]} \bigr) \rightarrow0
\]
in $P_{\SS_{[n]}|G_0}^{m}$-probability
as $m \rightarrow\infty$ and $n = n(m) \rightarrow\infty$ suitably.
Here, $P_{\SS_{[n]}|G_0}^{m}$ denotes\vspace*{-2pt} the true distribution
associated with population polytope $G_0$ that
generates a given $m\times n$ data\vspace*{2pt} set $\SS_{[n]}^{[m]}$. As mentioned,
$\delta_{m,n}$ is also the posterior contraction rate for the extreme points
among population structure variables $\vectheta_1,\ldots,\vectheta_k$.

\subsection*{Overview of results}
Suppose that $n \rightarrow\infty$ at a rate constrained by
$\log m < n$ and $\log n = o(m)$.
In an overfitted setting, that is, when the true population polytope
may have
less than $k$ extreme points, we show that under some mild
identifiability conditions
the posterior contraction rate
in either Hausdorff or minimum-matching distance metric is
$\delta_{m,n} \asymp [\frac{\log m}{m} + \frac{\log n}{n} +
\frac{\log n}{m}
]^{\afrac{1}{2(p+\alpha)}}$, where $p = (k-1)\wedge d$ is the
intrinsic dimension of the population polytope while
$\alpha$ denotes the regularity level near boundary of the
support of the density function for $\veceta$ (to be defined in sequel).
On the other hand, if either the true population polytope is known to have
exactly $k$ extreme points, or if the pairwise distances among
the extreme points are bounded from below by a known positive constant,
then the contraction rate is improved to a parametric rate
$\delta_{m,n} \asymp [\frac{\log m}{m} + \frac{\log n}{n} +
\frac{\log n}{m}
]^{\afrac{1}{2(1+\alpha)}}$.

The constraints on $n = n(m)$,
and the appearance of
quantity $\log n/m$ in the convergence rate are quite interesting. Both
the constraints
and the derived rate are rooted in a condition
on the required thickness of the prior support of the marginal
densities of
the data and an upper bound on the entropy
of the space of such densities. This suggests an interesting
interaction between layers in the latent hierarchy of the admixture model
worthy of further investigation. For instance,
it is not clear whether posterior consistency continues to hold
if $n$ falls outside of the specified range, and what effects this has
on convergence rates, with or without additional assumptions on the data.
This appears quite difficult with our present set of techniques.

We also establish minimax lower bounds for both settings. In the
overfitted setting, the obtained lower bound is
$(mn)^{-1/(q+\alpha')}$, where
$q = \lfloor k/2 \rfloor\wedge d$, and $\alpha'$ is a non-negative
constant to be defined in the sequel that satisfies $\alpha' \leq
\alpha$.
This lower bound can be strengthened with additional conditions on the model.
Although this lower bound does not match exactly
with the posterior contraction rate, both are notably nonparametrics-like
for depending on dimensionality $d$ and on $k$.
In particular, if $n \asymp m$, and $k\geq2d$, the posterior contraction
rate becomes $(\log m/m)^{-\afrac{1}{2(d+\alpha)}}$.
Compare this to the lower bound $m^{-2/(d+\alpha')}$, whose exponent
differs approximately by only a factor of 4 for large~$d$.
%It is also worth
%noting that if the distribution for the frequency
%vector $\veceta$ is uniform, i.e., $\alpha= 0$, we obtain an improved
%minimax lower bound of the order $m^{-1/q}$, which does not depend on
%$n$.

\comment{
As mentioned earlier the existing works on admixture models
include the recent papers by Arora et al \cite{Arora-LDA-12} and
Anandkumar et al \cite{Anandkumar-LDA-12}.
Both sets of authors analyzed specific learning algorithms for
recovering the population structure by taking
the viewpoint of matrix factorization. They both work on the
setting where the number of extreme points $k$ is known, and
$k \ll d$. Arora et al \cite{Arora-LDA-12} additionally
required interesting but very special conditions on the nature of
the extreme points, for which a polynomial time learning algorithm exists,
and established an estimation error rate for the algorithm.
Anandkumar et al \cite{Anandkumar-LDA-12} proposed a novel moment-based
estimation method and obtained a consistency result. By contrast,
we analyze general Bayesian
estimation without concerning a specific inference algorithm.
%(This goes without saying that under general conditions
%the posterior contraction entails
%convergence of procedures such as the maximum likelihood
%estimation method).
The posterior contraction rates and minimax results
obtained in this paper
appear new. The posterior asymptotics and convex geometric techniques developed
here are quite distinct from all existing works.
}

\subsection*{Method of proofs and tools}
The general
framework of posterior asymptotics for density estimation
has been well-established
\cite{Ghosal-Ghosh-vanderVaart-00,Shen-Wasserman-01}
(see also \cite
{Barron-Shervish-Wasserman-99,Ghosh-Ramamoorthi-02,Walker-04,Ghosal-vanderVaart-07,Walker-Lijoi-Prunster-07}).
This framework continues
to be very useful, but the analysis of mixing
measure estimation in multi-level models presents distinct new challenges.
In Section~\ref{Sec-gen}, we shall formulate an abstract theorem
(Theorem~\ref{Thm-Gen})
on posterior contraction of latent variables of interest in an admixture
model, given $m\times n$ data, by reposing on the framework
of \cite{Ghosal-Ghosh-vanderVaart-00} (see also \cite{Nguyen-11}).
The main novelty here is
that we work on the space of latent variables (e.g., space
of latent population structures endowed with Hausdorff or a comparable metric)
as opposed to the space of data densities endowed with Hellinger metric.
A basic quantity is the \emph{Hellinger information}
of the Hausdorff metric for a given subset of polytopes.
Indeed, the Hellinger information
is a fundamental quantity running through the analysis, which
ties together the amount of data $m$ and $n$ -- key quantities that are
associated with different levels in the model hierarchy.

The bulk of the paper is devoted to establishing properties of
the Hellinger information, which are fed into
Theorem~\ref{Thm-Gen} so as to obtain concrete convergence rates.
This is achieved through a number of
inequalities which illuminate the relationship between Hausdorff
distance
of a given pair of population polytopes $G,G'$,
and divergence functionals (e.g., Kullback--Leibler divergence or
variational distance) of the induced marginal data densities.
The technical challenges lie in the fact that in order to relate
$G$ to the marginal density of the data, one has to integrate out
multiple layers of latent variables, $\veceta$ and $\vecbeta$.
Techniques in convex geometry come in very handily in the derivation
of both lower and upper bounds \cite{Schneider-93}.

The remainder of the paper is organized as follows. The model
and main results are described in Section~\ref{Sec-statement}.
Section~\ref{Sec-geometry} describes the basic geometric assumptions
and their consequences. An abstract theorem for posterior contraction
for $m\times n$ data setting is formulated in Section~\ref{Sec-gen},
whose conditions are
verified in the subsequent sections. Section~\ref{Sec-contraction}
presents inequalities for Hausdorff distances which result in
lower bounds on the Hellinger information, while Section~\ref{Sec-concentration}
provides a lower bound on Kullback--Leibler neighborhoods
of the prior support (that is, a bound the prior thickness).
Proofs of main theorems and other technical lemmas
are presented in Section~\ref{Sec-proofs} and the \hyperref[appA]{Appendices}.

%pa1.subsection.subsubsection.1 #&#
\subsection*{Notations}
$B_{\d}(\vectheta,r)$ denotes a closed $\d$-dimensional
Euclidean ball centered at $\vectheta$ and has radius $r$.
$G_\epsilon$ denotes the Minkowsky sum $G_\epsilon:= G + B_{d+1}(\vec
{0},\epsilon)$.
$\bd G, \extr G, \operatorname{Diam}G$,
$\aff G, \vol_{p} G$ denote the boundary, the set of extreme points,
the diameter, the affine span, and the $p$-dimensional volume of
set $G$, respectively. ``Extreme points''
and ``vertices'' are interchangeable throughout this paper.
We define the dimension of a convex polytope to be the dimension of its
affine hull. It is a well-known fact that if a polytope has $k$ extreme
points in general positions in $\Delta^d$, then its dimension is
$(k-1)\wedge d$.
Set-theoretic difference between two sets is defined as $G
\bigtriangleup
G' = (G\setminus G') \cup(G' \setminus G)$.
$N(\epsilon, \Gcal, d_{\mathcal{H}})$ denotes the covering number of
$\Gcal$ in
Hausdorff metric $d_{\mathcal{H}}$.
$D(\epsilon, \Gcal, d_{\mathcal{H}})$ is the packing number of
$\Gcal$ in
Hausdorff metric.
Several divergence measures for probability distributions are employed:
$K(p,q), h(p,q), V(p,q)$ denote
Kullback--Leibler divergence, Hellinger and total variation distance
between two densities $p$ and $q$ defined with respect to a measure
on a common space: $K(p,q) = \int p\log(p/q) $,
$h^2(p,q) = \frac{1}{2}\int(\sqrt{p} - \sqrt{q})^2$
and $V(P,Q) = \frac{1}{2}\int|p -q|$.
In addition, we define $K_2 = \int p[\log(p /q)]^2$.
%Several probability distributions are analyzed throughout the paper:
%$P_{\vecbeta}, \Peta, P_{\veceta\times\SSn|G}$ are the distribution
%of $\vecbeta$,
%the distribution of $\veceta$ given $G$, and the joint distribution of
%$\veceta$ and an $n$-vector $\SSn$ given $G$,
%respectively. $\PG$ is the marginal density of $\SSn$ given $G$
%(by having $\veceta$ integrated out). $P_{\SSn|G}^m$ denotes the
%product
%distribution for the full data set $\SSn^{[m]}$.
%The lower-case $p_{\vecbeta}, p_{\veceta|G}, p_{\veceta\times\SSn|G},
%p_{\SSn|G}$, $p_{\SSn|G}^m$ are the corresponding
%densities.
Throughout the paper, $f(m,n,\epsilon) \lesssim g(m,n,\epsilon)$,
equivalently, $f = \mathrm{O}(g)$,
means $f(m,n,\epsilon) \leq Cg(m,n,\epsilon)$ for some constant
$C$ independent of asymptotic quantities $m, n$ and $\epsilon$ --
details about the dependence of $C$ are made explicit unless obvious
from the context. Similarly,
$f(m,n,\epsilon) \gtrsim g(m,n,\epsilon)$ or $f = \Omega(g)$
means $f(m,n,\epsilon) \geq Cg(m,n,\epsilon)$.

%s2 #&#
\section{Main results}
\label{Sec-statement}

%pa2.subsection.subsubsection.1 #&#
\subsection*{Model description}
As mentioned in the \hyperref[intro]{Introduction},
the central objects of the admixture model are \emph{population structure}
variables $(\vectheta_1,\ldots, \vectheta_k)$, whose convex hull is called
the \emph{population polytope}: $G = \conv(\vectheta_1,\ldots,
\vectheta_k)$.
$\vectheta_1,\ldots, \vectheta_k$ reside in
$d$-dimensional probability simplex $\Delta^{d}$.
$k < \infty$ is assumed known.
Note that $G$ has at most $k$ vertices (i.e., extreme points) among
$\vectheta_1, \ldots, \vectheta_k$.

A random vector $\veceta\in G$ is parameterized
by $\veceta= \beta_1 \vectheta_1 + \cdots+ \beta_k \vectheta_k$, where
$\vec{\beta} = (\beta_1,\ldots, \beta_k) \in\Delta^{k-1}$ is a
random vector distributed according to a distribution
$P_{\vecbeta|\gamma}$ for some parameter $\gamma$ (both \cite
{Pritchard-etal-00}
and \cite{Blei-etal-03} used the Dirichlet distribution).
Given $\vectheta_1,\ldots,\vectheta_k$, this induces a probability
distribution
$P_{\veceta|G}$ whose support is the convex set $G$. Details
of this distribution, suppressed for the time being,
are given explicitly by Equations \eqref{Eqn-density-1} and \eqref{Eqn-density-2}.
[To be precise, $P_{\veceta|G}$ should be written as
$P_{\veceta|\vectheta_1,\ldots,\vectheta_k; G}$.
That is, $G$ is always attached with a specific set of $\vectheta_j$'s.
Throughout the paper, this specification of $G$ is always understood
but notationally suppressed to avoid cluttering.]

For each individual $i=1,\ldots,m$, let $\veceta_i \in\Delta^d$ be
an independent random vector distributed by $P_{\veceta|G}$. The
observed data
associated with $i$, $\SS_{[n]}^i = (X_{ij})_{j=1}^{n}$ are assumed to
be i.i.d.
draws from the multinomial distribution $\Mult(\veceta_i)$
specified by $\veceta_i := (\eta_{i0},\ldots,\eta_{id})$.
That is, $X_{ij} \in\{0,\ldots,d\}$ such that
$P(X_{ij} = l | \veceta_i) = \eta_{il}$ for $l=0,\ldots,d$.

Admixture models are simple when specified in a hierarchical manner as
given above.
The relevant distributions are written down below.
The joint distribution of the generic random variable
$\veceta$ and $n$-vector $\SS_{[n]}$ (dropping
superscript $i$ used for indexing a specific individual) is denoted
by $P_{\veceta\times\SS_{[n]}|G}$ and its density $p_{\veceta\times
\SS_{[n]}
| G}$.
We have
%
%e1 #&#
\begin{equation}
\label{Eqn-joint-eta-SSn} p_{\veceta\times\SS_{[n]}|G}\bigl(\veceta_i,\SS_{[n]}^i
\bigr) = p_{\veceta|G}(\veceta_i) \times\prod
_{j=1}^{n}\prod_{l=0}^{d}
\eta_{il}^{\indicator(X_{ij} = l)}.
\end{equation}
The distribution of $\SS_{[n]}$, denoted by $P_{\SS_{[n]}|G}$, is
obtained by integrating out $\veceta$, which
yields the following density with respect to
counting measure:
%
%e2 #&#
\begin{equation}
\label{Eqn-marginal-SSn} p_{\SS_{[n]}|G}\bigl(\SS_{[n]}^{i}\bigr) =
\int_{G} \prod_{j=1}^{n}
\prod_{l=0}^{d}\eta _{il}^{\indicator(X_{ij} = l)}
\,\mathrm{d} P_{\veceta|G}(\veceta_i).
\end{equation}
The joint distribution of the full data set $\SS_{[n]}^{[m]}:= (\SS
_{[n]}^i)_{i=1}^{m}$,
denoted by $P_{\SS_{[n]}|G}^m$, is a product distribution:
%
%e3 #&#
\begin{equation}
\label{Eqn-marginal-SSmn} P_{\SS_{[n]}|G}^m\bigl(\SS_{[n]}^{[m]}
\bigr) := \prod_{i=1}^{m} P_{\SS
_{[n]}|G}
\bigl(\SS_{[n]}^{i}\bigr).
\end{equation}

Admixture models are customarily introduced in an equivalent way as
follows \cite{Blei-etal-03,Pritchard-etal-00}:
For each $i=1,\ldots,m$, draw an independent random variable $\vec
{\beta} \in
\Delta^{k-1}$ as $\vec{\beta} \sim P_{\vecbeta|\gamma}$.
Given $i$ and $\vec{\beta}$, for $j=1,\ldots,n$, draw $Z_{ij} |\vec
{\beta}
\stackrel{\mathrm{i.i.d.}}{\sim} \Mult(\vec{\beta})$. $Z_{ij}$ takes
values in $\{1,\ldots,k\}$. Now, data point $X_{ij}$ is randomly
generated by $X_{ij} | Z_{ij} = l, \vectheta
\sim\Mult(\vectheta_l)$. This yields the same joint distribution
of $\SS_{[n]}^i = (X_{ij})_{j=1}^{n}$ as the one described earlier.
The use of latent variables $Z_{ij}$ is amenable to the development
of computational algorithms for inference. However, this representation
bears no significance within the scope of this work.

%pa2.subsection.subsubsection.2 #&#
\subsection*{Asymptotic setting and metrics on population polytopes}
Assume the data set $\SS_{[n]}^{[m]}= (\SS_{[n]}^i)_{i=1}^m$ of size
$m\times n$
is generated according an admixture model given by
``true'' parameters $\vectheta_1^*,\ldots, \vectheta_k^*$.
$G_0 =
\conv(\vectheta_{1}^*,\ldots, \vectheta_{k}^*)$ is
the true population polytope.
Under the Bayesian estimation framework,
the population structure variables $(\vectheta_1,\ldots, \vectheta_k)$
are random and endowed with a prior distribution $\Pi$.
The main question
to be addressed in this paper is the contraction behavior of
the posterior distribution $\Pi(G|\SS_{[n]}^{[m]})$, as
the number of data points $m\times n$ goes to infinity.

It is noted that we do not always assume that the number
of extreme points of the population polytope $G_0$
is $k$. We work in a general overfitted setting where
$k$ only serves as the upper bound of the true
number of extreme points for the purpose of model parameterization.
The special case in which
the number of extreme points of $G_0$ is known a priori is also
interesting and will be considered.

\comment{
\begin{itemize}
\item[(1)] What is the contraction behavior of
$\Pi(\vectheta_1,\ldots,\vectheta_k | \eta_1,\ldots,\eta_m)$ as
$m\rightarrow\infty$.
\item[(2)] In a full admixture model, we do not have access to $\vec
{\eta}$.
Instead, we observe $X_{ij}$'s. As both $m \rightarrow\infty$, $n
\rightarrow\infty$,
what is the contraction for the posterior probability of the latent population
structure
\[
\Pi(G|X_{ij}; i=1,\ldots,m; j=1,\ldots,n).
\]
\item[(3)] Above it is assumed that the probability distribution $\Pi
(\vec{\beta}|\gamma)$
is given, i.e., $\gamma= \gamma^*$ is known.
What if $\gamma^*$ is unknown? Can we still learn $\vectheta_1,\ldots
,\vectheta_k$?
Here we will need to endow a prior on $\gamma$, $\Pi(\gamma)$, and analyze
the posterior $\Pi(G,\gamma|X_{ij}; i=1,\ldots,m; j=1,\ldots,n)$ jointly.
\end{itemize}
}

Let $\extr G$ denote the set of extreme points of a given polytope $G$.
$\Gcal^k$ is the set of population polytopes in $\Delta^d$ such that
$|\extr G| \leq k$. Let $\mathcal{G}^*= \bigcup_{2\leq k < \infty}
\Gcal^k$
be the set of
population polytopes that have finite number of extreme points in
$\Delta^d$.
A natural metric on $\mathcal{G}^*$ is the following ``minimum-matching''
Euclidean distance:
\[
d_{\mathcal{M}}\bigl(G,G'\bigr) = \max_{\vectheta\in\extr G}
\min_{\vectheta
' \in\extr G'} \bigl \|\vectheta- \vectheta'\bigr \| \vee \max
_{\vectheta' \in\extr G'} \min_{\vectheta\in\extr G} \bigl \| \vectheta'
- \vectheta\bigr \|.
\]
%
%where the minimum is taken over all permutations of $(1,\ldots,k)$.
A more common metric is the Hausdorff metric:
\[
d_{\mathcal{H}}\bigl(G,G'\bigr) = \min\bigl\{\epsilon\geq0 | G
\subset G_{\epsilon}'; G' \subset
G_{\epsilon}\bigr\} = \max_{\vectheta\in G} d\bigl(
\vectheta,G'\bigr) \vee\max_{\vectheta' \in
G'} d\bigl(
\vectheta',G\bigr).
\]
Here, $G_{\epsilon} = G+ B_{d+1}(\vec{0},\epsilon) := \{\vectheta
+e| \vectheta\in G,
e \in\real^{d+1},
\|e\| \leq1\}$,
and $d(\vectheta,G') := \inf\{\|\vectheta-\vectheta'\|, \vectheta'
\in G'\}$.
Observe that $d_{\mathcal{H}}$ depends on the boundary structure of sets,
while $d_{\mathcal{M}}$ depends on only extreme points.
In general, $d_{\mathcal{M}}$ dominates $d_{\mathcal{H}}$, but under
additional
mild assumptions the two metrics are equivalent (see Lemma~\ref{Lem-M}).

We introduce a notion of regularity for a family probability distributions
defined on convex polytopes $G \in\mathcal{G}^*$. This notion is
concerned with
the behavior near the boundary of the support of distributions
$P_{\veceta|G}$.
We say a family of distributions $\{P_{\veceta|G}| G\in\Gcal^k\}$
is $\alpha$-regular if for any $G \in\Gcal^k$ and any $\veceta_0
\in\bd G$,
%
%e4 #&#
\begin{equation}
\label{Eqn-regularity} P_{\veceta|G}\bigl(\|\veceta- \veceta_0\| \leq\epsilon\bigr)
\geq c \epsilon^\alpha\vol_{p} \bigl(G \cap B_{d+1}(
\veceta_0,\epsilon)\bigr),
\end{equation}
where $p$ is the number of dimensions of the affine space $\aff G$ that spans
$G$, constant $c>0$ is independent of $G, \veceta_0$ and $\epsilon$.

According to Lemma~\ref{Lem-Dirichlet} (in Section~\ref{Sec-concentration})
$\alpha$-regularity
holds for a range of $\alpha$, when $P_{\vecbeta|\gamma}$ is a
Dirichlet distribution, but there may be
other choices. We will see that $\alpha$ plays an important role
in characterize the rates of contraction for the posterior of the
population polytope.
%
%Let $\PcalAlpha$ be the set of all $\alpha$-regular distributions
%supported by $G$, for all $G\in\Gstar$.
%
%pa2.subsection.subsubsection.3 #&#
\begin{Assumptions*} $\Pi$ is a prior distribution
on $\vectheta_1,\ldots, \vectheta_k$ such that
the following hold for the relevant parameters that reside
in the support of $\Pi$:
\begin{enumerate}[(S3b)]
\item[(S0)] Geometric properties (\textup{\ref{propA1}}) and (\textup{\ref{propA2}}) listed in Section~\ref{Sec-geometry}
are satisfied uniformly for all $G$.

\item[(S1)] Each of $\vectheta_1,\ldots,\vectheta_k$
is bounded away from the boundary of $\Delta^d$. That is, if
$\vectheta_{j} = (\theta_{j,0},\ldots, \theta_{j,d})$ then
$\min_{l=0,\ldots,d} \theta_{j,l} > c_0$ for all $j=1,\ldots,k$.
\item[(S2)] For any small $\epsilon$,
$\Pi(\|\vectheta_j-\vectheta_j^*\| \leq\epsilon\ \forall j=1,\ldots,k)
\geq c'_0\epsilon^{kd}$, for some $c'_0>0$.

\comment{
\item[(S3)] $\vecbeta= (\beta_1,\ldots, \beta_k)$ is distributed
(a priori) according to a symmetric
probability distribution $P_{\vecbeta}$ on $\Delta^{k-1}$. That is,
the random variables $\beta_1,\ldots,\beta_k$ are exchangeable.
}

\item[(S3a)] $P_{\vecbeta}$ induces a family of distributions $\{
P_{\veceta|G}|
G\in\Gcal^k\}$ that is $\alpha$-regular.

\item[(S3b)] $\vecbeta = (\beta_1,\ldots, \beta_k)$ is
distributed
(a priori) according to a symmetric
probability distribution $P_{\vecbeta}$ on
$\Delta^{k-1}$. That is,
the random variables $\beta_1,\ldots,\beta_k$ are a priori
exchangeable.
\end{enumerate}
\end{Assumptions*}

%th1 #&#
\begin{theorem}
\label{Thm-Main} Let $G_0 \in\Gcal^k$ and $G_0$ is in the support
of prior $\Pi$. Let $\d= (k-1)\wedge d$.
Under assumptions \textup{(S0)--(S3)} of the admixture model, as
$m\rightarrow\infty$ and $n\rightarrow\infty$ such as
$\log\log m \leq\log n = \mathrm{o}(m)$,
for some sufficiently large constant $C$ independent of $m$ and $n$,
%
%e5 #&#
\begin{equation}
\label{Eqn-PC} \Pi\bigl(d_{\mathcal{M}}(G_0,G) \geq C
\delta_{m,n} | \SS_{[n]}^{[m]}\bigr) \longrightarrow0
\end{equation}
in $P_{\SS_{[n]}|G_0}^m$-probability. Here,
\[
\delta_{m,n} = \biggl[\frac{\log m}{m} + \frac{\log n}{n} +
\frac{\log n}{m} \biggr]^{\afrac{1}{2(p+\alpha)}}.
\]
The same statement holds for the Hausdorff metric $d_{\mathcal{H}}$.
\end{theorem}

%pa2.subsection.subsubsection.4 #&#
\begin{Remarks*}
\begin{enumerate}[6.]
\item[1.] Geometric assumption (S0) and its consequences
are presented in the next section. (S0), (S1) and (S2) are
mild assumptions observed in practice
(cf. \cite{Blei-etal-03,Pritchard-etal-00}).

\item[2.] The assumption in (S3b) that $P_{\vecbeta}$ is symmetric is
relatively
strong, but it has been widely adopted in practice
(e.g., symmetric Dirichlet
distributions, including the uniform distribution).
This technical condition is not intrinsic
to the theory,
and is required only to establish an upper bound
of the Kullback--Leibler distance
in terms of Hausdorff distance.
In fact, it may be replaced if such an upper bound can be
established by
some other means. See also the remark following the statement of
Lemma~\ref{Lem-W}.
\comment{
\item The assumption in (S3) that $P_{\vecbeta}$ is symmetric is relatively
strong, but again it has been widely adopted (e.g., symmetric Dirichlet
distributions, including the uniform distribution).
It may be difficult to try to relax this assumption if one insists on
using Hausdorff metric, see the remark following the statement of Lemma~\ref{Lem-W}.
Nonetheless, (S3) is only used for establishing the thickness of the prior
(see Theorem~\ref{Thm-KL}), an important technical component of the
Bayesian asymptotics.
(S3) is not required if one is primarily interested in
establishing convergence rates for point estimators.
}

\item[3.] In practice, $P_{\vecbeta}$ may be further
parameterized as $P_{\vecbeta|\gamma}$, where $\gamma$ is endowed
with a prior distribution. Then, it would be of interest to also
study the posterior contraction behavior for $\gamma$.
In this paper, we have opted to focus only on convergence
behavior of the population structure to simplify the exposition and
the results.

\item[4.] The appearance of both $m^{-1}$ and $n^{-1}$ in the contraction
rate suggests that if either $m$ or $n$ is small, the rate would suffer
even if the total amount of data $m\times n$ increases.
What is quite interesting is the
appearance of $\log n/m$. This is
rooted in an entropy condition (cf. Theorem \ref{Thm-Gen} in Section \ref{Sec-gen}),
which requires an upper bound of the KL
divergence in terms
of Hausdorff distance. It is possible that the appearance of
$\log n/m$ is due to
our general proof technique of posterior contraction
presented in Section \ref{Sec-gen}.
From a hierarchical
modeling viewpoint, this result highlights an interesting interaction
of sample sizes provided to different levels in
the model hierarchy. This issue has not been widely discussed in the
hierarchical modeling literature in a theoretical manner, to
the best of our knowledge.

\item[5.] Note the constraints that $n > \log m$ and $\log n=\mathrm{o}(m)$ are required
in order to obtain rates of posterior contraction.
These constraints are related to the term $\log n /m$ mentioned above --
they stem from
the upper bound on Kullback--Leibler in Lemma~\ref{Lem-KL-bound}. The remark
following the statement of this lemma explains why the upper bound almost
always grows with $n$.
A~very special situation is presented in Lemma~\ref{Lem-KL-ez}
where an upper bound on Kullback--Leibler distance can be obtained that
is independent of $n$. However, such a situation
cannot be verified in any reasonable estimation setting. This suggests
that with our proof technique, we almost always require $n$ to grow
at a constrained rate relatively to $m$ in order to obtain posterior
contraction rates.

\item[6.] Constant $\alpha$ plays an important role in the rate exponent.
Intuitively, the larger $\alpha$ is, the weaker the guaranteed probability
mass accrued near the boundary of the population polytope, which implies
less data observed for points located near the boundary. This entails a
weaker guarantee on the rate of convergence. Indeed,
a key step in the proof of the theorem is that Equation \eqref{Eqn-regularity}
enables us to transfer an upper bound on the diminishing variational
distance between, say
distributions $P_{\veceta|G}$ and $P_{\veceta|G'}$, to an upper bound
on the
Hausdorff distance between $G$ and $G'$, while incurring an
extra term $\alpha$ in the exponent.

\item[7.] The exponent $\frac{1}{2(p+\alpha)}$ suggests a slow,
nonparametric-like convergence rate. Moreover,
later in Theorem~\ref{Thm-minimax} we show that this is qualitatively
quite close to a minimax lower bound. On the other hand,
the following theorem shows that it is possible to achieve a
parametric rate if additional constraints are imposed on
the true $G_0$ and/or the prior $\Pi$:
\end{enumerate}
\end{Remarks*}

%th2 #&#
\begin{theorem}
\label{Thm-Main-2}
Let $G_0 \in\Gcal^k$ and $G_0$ is in the support of prior $\Pi$.
Assume \textup{(S0)--(S3a), (S3b)}, and either one of the
following two conditions hold:
\begin{enumerate}[(b)]
\item[(a)] $|\extr G_0| = k$, or
\item[(b)] There is a known constant $r_0>0$ such that
the pairwise distances of the extreme points of all $G$ in the support
of the prior
are bounded from below by $r_0$.
\end{enumerate}
Then, as $m\rightarrow\infty$ and $n\rightarrow\infty$ such that
$\log m < n$ and $\log n=\mathrm{o}(m)$,
Equation \eqref{Eqn-PC} holds with
\[
\delta_{m,n} = \biggl[\frac{\log m}{m} + \frac{\log n}{n} +
\frac{\log n}{m} \biggr]^{\afrac{1}{2(1+\alpha)}}.
\]
The same statement holds for the Hausdorff metric $d_{\mathcal{H}}$.
\end{theorem}

The next theorem produces minimax lower bounds that are qualitatively quite
similar to the nonparametric-like rates obtained in Theorem~\ref{Thm-Main}.
In the following theorem, $\veceta$ is not parameterized by
$\vecbeta$ and $\vectheta_j$'s as in the admixture model.
Instead, we shall simply replace assumptions (S3a) and (S3b) on
$P_{\vecbeta|\gamma}$ by either one of the following assumptions on
$P_{\veceta|G}$:
\begin{enumerate}[(S4$^{\prime}$)]
%$\Peta$ satisfies the following property:
%for any $\veceta_0 \in\bd G$,
%
\item[(S4)] There is a non-negative constant $\alpha'$ such that
for any pair of $p$-dimensional polytopes $G' \subset G$ that satisfy
Property \ref{propA1},
\[
V(P_{\veceta|G}, P_{\veceta|G'}) \lesssim d_{\mathcal
{H}}
\bigl(G,G'\bigr)^{\alpha'} \vol_p G\setminus
G'.
\]

\item[(S4$^{\prime}$)] For any $p$-dimensional polytope $G$,
$P_{\veceta|G}$ is the uniform distribution on $G$.
\end{enumerate}

Note that the condition of $\alpha$-regularity (cf. Equation \eqref
{Eqn-regularity})
implies that $\alpha\geq\alpha'$.
In particular, if (S4$^{\prime}$) is satisfied, then both (S3a)/(S3b) and
(S4) hold with $\alpha= \alpha'= 0$.

Since a parameterization for $\veceta$ is not needed, the overall
model can be simplified as follows: Given population polytope $G \in
\Delta^d$,
for each $i=1,\ldots, m$, draw $\veceta_i \stackrel{\mathrm{i.i.d.}}{\sim}
P_{\veceta|G}$.
For each $j=1,\ldots, n$, draw $\SS_{[n]}^i = (X_{ij})_{j=1}^{n}
\stackrel
{\mathrm{i.i.d.}}{\sim}
\Mult(\veceta_i)$.

%th3 #&#
\begin{theorem}
\label{Thm-minimax}
Suppose that $G_0 \in\Gcal^k$ satisfies assumptions \textup{(S0)}, \textup{(S1)} and \textup{(S2)}.
Point estimates $\hatG= \hat{G}(\SS_{[n]}^{[m]})$ take value in the set
$\Gcal^*$.
In the following,
the multiplying constants in $\gtrsim$ depend only
on constants specified by these assumptions.
\begin{enumerate}[(b)]
\item[(a)] Let $q = \lfloor k/2 \rfloor\wedge d$.
Under assumption \textup{(S4)}, we have
\[
\inf_{\hat{G} \in\Gcal^*} \sup_{G_0 \in\Gcal^k} P_{\SS
_{[n]}|G_0}^m
d_{\mathcal{H}} (G_0, \hat{G}) \gtrsim \biggl( \frac{1}{mn}
\biggr)^{\afrac{1}{q+\alpha'}}.
\]
\item[(b)] Let $q = \lfloor k/2 \rfloor\wedge d$.
Under assumption \textup{(S4$^{\prime}$)}, we have
\[
\inf_{\hat{G} \in\Gcal^*} \sup_{G_0 \in\Gcal^k} P_{\SS
_{[n]}|G_0}^m
d_{\mathcal{H}} (G_0, \hat{G}) \gtrsim \biggl( \frac{1}{m}
\biggr)^{\sfrac{1}{q}}.
\]
\item[(c)] Assume \textup{(S4)}, and that either condition \textup{(a)} or \textup{(b)} of
Theorem~\ref{Thm-Main-2} holds, then
\[
\inf_{\hat{G} \in\Gcal^*} \sup_{G_0 \in\Gcal^k} P_{\SS
_{[n]}|G_0}^m
d_{\mathcal{H}} (G_0, \hat{G}) \gtrsim \biggl( \frac{1}{mn}
\biggr)^{\afrac{1}{1+\alpha'}}.
\]
Furthermore, if \textup{(S4)} is replaced by \textup{(S4$^{\prime}$)}, the lower bound becomes $1/m$.
\end{enumerate}
\end{theorem}

%pa2.subsection.subsubsection.5 #&#
\begin{Remarks*}
\begin{enumerate}[3.]
\item[1.] There is a gap between the posterior contraction rate in
Theorem~\ref{Thm-Main}
and the minimax lower bound in Theorem~\ref{Thm-minimax}(a). This is
expected because the infimum over point estimates $\hatG$ is taken
over $\Gcal^*$, as opposed to $\Gcal^k$. Nonetheless, the lower
bounds are notably dependent on $d$ and $k$, thereby provide a
partial justification for the
nonparametics-like posterior contraction rates.
It is interesting to note that
if $k \geq2d \gg\alpha$, and allowing $m \asymp n$,
the rate exponents differ approximately by only a factor of 4.
That is, $m^{-1/2(d+\alpha)}$ vis-\`a-vis $m^{-2/(d+\alpha')}$.

\item[2.] The nonparametrics-like lower bounds in part (a) and (b) in the overfitted
setting are somewhat
surprising even if $P_{\vecbeta}$ is known exactly (e.g., $P_{\vecbeta
}$ is uniform distribution).
Since we are more likely to be in the overfitted setting than knowing
the exact number of extreme points, an implication of
this is that it is important in practice to
impose a lower bound on the pairwise
distances between the extreme points of the population polytope.

\item[3.] The results in part (b) and (c) under assumption (S4$'$)
present an interesting scenario in which the obtained lower bounds
do not depend on $n$, which determines the amount of data at the
bottom level in the model hierarchy.

\comment{
(iv) It is worth mentioning that the exponent for $m$ in
the lower bounds $(1/mn)^{1/(d+\alpha)}$ of part (a) (when $k \geq2d$)
and $(1/m)^{1/d}$ of part (b) (when $k \geq2d$ and (S4') holds)
is reminiscent of a general minimax
optimal rate $(\log m/m)^{1/(d+\alpha)}$
for estimating the support of the density function $
{p_{\veceta|G}}$, assuming
that an i.i.d. $m$-sample of $\veceta$ is \emph{directly} observed.
\cite{Tsybakov-97,Singh-etal-09}. A word of caution about making
this comparison is that while the latter problem is easier due to
the direct observations of $\veceta$, the density support for $\veceta$
is not required to be convex as is the case with admixture models.
}
\end{enumerate}
\end{Remarks*}
%s3 #&#
\section{Geometric assumptions and basic lemmas}
\label{Sec-geometry}

In this section, we discuss the geometric assumptions postulated in
the main theorems, and describe their consequences using
elementary arguments in convex geometry of Euclidean spaces.
These results relate Hausdorff metric, the minimum-matching metric, and
the volume of the set-theoretic difference of polytopes.
These relationships prove crucial in obtaining explicit
posterior contraction rates. Here, we state the properties
and prove the results for $p$-dimensional polytopes and convex bodies
of points in $\Delta^d$, for a given $p \leq d$. (Convex bodies are
bounded convex sets that may have an unbounded number of extreme points.
Within this section, the detail of the ambient space is irrelevant. For
instance,
$\Delta^d$ may be replaced by $\real^{d+1}$ or a higher dimensional
Euclidean space.)

\renewcommand{\theProperty}{A\arabic{Property}}
%pa3.subsection.subsubsection.1 #&#
\begin{Property} [(Property of thick body)]\label{propA1} For some $r,R > 0$,
$\vectheta_c \in\Delta^{d}$,
$G$ contains the spherical ball $B_{\d}(\vectheta_c,r)$ and is
contained in
$B_{\d}(\vectheta_c,R)$.
\end{Property}

%pa3.subsection.subsubsection.2 #&#
\begin{Property}[(Property of non-obtute corners)]\label{propA2}
For some small $\delta> 0$, at each vertex of $G$ there is a supporting
hyperplane whose angle formed with any edges adjacent to that vertex is
bounded from below by $\delta$.
\end{Property}
%For some $\delta> 0$, pairwise distances between every pair of
%vertices
%are bounded from below by $\delta$.

We state key geometric lemmas that will be used throughout the paper.
Bounds such as those given by Lemma~\ref{Lem-diff-1}
are probably well-known in the folklore of convex
geometry (for instance, part (b) of that lemma is similar to (but not
precisely the same as)
Lemma~2.3.6. from \cite{Schneider-93}).
Due to the absence of direct references, we include the proof of this and
other lemmas in the \hyperref[appA]{Appendix}.
%
%, while part (a) may
%be extracted from an argument in \cite{Dumbgen-Walther-96} (Corollary
%1).

%le1 #&#
\begin{lemma}
\label{Lem-M}
\begin{enumerate}[(b)]
\item[(a)] $d_{\mathcal{H}}(G,G') \leq d_{\mathcal{M}}(G,G')$.

\item[(b)] If the two polytopes $G,G'$ satisfy Property \textup{\ref{propA2}},
then $d_{\mathcal{M}}(G,G') \leq C_0 d_{\mathcal{H}}(G,G')$, for
some positive constant $C_0 > 0$ depending only on $\delta$.
\end{enumerate}
\end{lemma}

According to part (b) of this lemma, convergence of a sequence of
convex polytope $G \in\Gcal^k$ to $G_0 \in\Gcal^k$ in
Hausdorff metric
entails the convergence of the extreme points of $G$ to
those of $G_0$. Moreover, they share the same rate as
the Hausdorff convergence.

%le2 #&#
\begin{lemma}
\label{Lem-diff-1}
There are positive constants $C_1$ and $c_1$ depending
only on $r,R,\d$ such that for any two $\d$-dimensional
convex bodies $G,G'$ satisfying Property \textup{\ref{propA1}}:
\begin{enumerate}[(b)]
\item[(a)] $\vol_{\d} G \bigtriangleup G' \geq c_1d_{\mathcal
{H}}(G,G')^{\d}$.
\item[(b)] $\vol_{\d} G \bigtriangleup G' \leq C_1d_{\mathcal{H}}(G,G')$.
\end{enumerate}
\end{lemma}

\comment{
%
%le3 #&#
\begin{lemma}
\label{Lem-diff-2}
If convex bodies $G,G'$ in $\real^d$ satisfy property \AA.
In addition, either $G$ or $G'$ is smooth. Then,
\[
\vol_{\d} G \bigtriangleup G' \geq c(r,d)
d_{\mathcal{H}}\bigl(G,G'\bigr)^{(d+1)/2}.
\]
\end{lemma}
\begin{pf}
Without loss of generality, we can assume that $G$ is a spherical ball.
\end{pf}
}
%pa3.subsection.subsubsection.3 #&#
\begin{Remark*} The exponents in both bounds in Lemma~\ref{Lem-diff-1}
are attainable. Indeed, for the lower bound in part (a), consider a
fixed convex polytope $G$. For each vertex $\vectheta_i \in G$,
consider point $x$ that lie on edges incident to $\vectheta_i$ such
that $\|x-\vectheta_i\| = \epsilon$. Let $G'$ be the convex hull of all
such $x$'s and the remaining vertices of $G$.
Clearly, $d_{\mathcal{H}}(G,G') = \mathrm{O}(\epsilon)$, and
$\vol_\d G\setminus G' \leq \mathrm{O}(\epsilon^{\d})$.
Thus, for the collection of convex polytopes $G'$ constructed in this way,
$\vol_\d(G\bigtriangleup G') \asymp d_{\mathcal{H}}(G,G')^{\d}$.
The upper bound in part (b) is also tight for a broad class of
convex polytopes, as exemplified by the following lemma.
\end{Remark*}

%le4 #&#
\begin{lemma}
\label{Lem-diff-3}
Let $G$ be a fixed polytope and $|\extr G| = k < \infty$.
$G'$ an arbitrary polytope in $\Gcal^*$.
Moreover, either one of the following conditions holds:
\begin{enumerate}[(b)]
\item[(a)] $|\extr G'| = k$, or
\item[(b)] The pairwise distances between
the extreme points of $G'$ is bounded away from a constant $r_0>0$.
\end{enumerate}
Then, there is a positive constant $\epsilon_0 = \epsilon_0(G)$
depending only on $G$,
a positive constant $c_2 = c_2(G)$ in case \textup{(a)} and $c_2 = c_2(G, r_0)$
in case \textup{(b)},
such that
\[
\vol_\d G \bigtriangleup G' \geq c_2
d_{\mathcal{H}}\bigl(G,G'\bigr)
\]
as soon as $d_{\mathcal{H}}(G,G') \leq\epsilon_0(G)$.
\end{lemma}

%pa3.subsection.subsubsection.4 #&#
\begin{Remark*} We note that the bound obtained in this lemma is
substantially stronger than the one obtained by Lemma~\ref{Lem-diff-1}
part (a).
This is due to the asymmetric roles of $G$, which is held fixed, and $G'$,
which can vary. As a result, constant $c_2$
as stated in the present lemma is independent of $G'$ but allowed to be
dependent on $G$.
By contrast, constant
$c_1$ in Lemma~\ref{Lem-diff-1} part (a) is independent of both $G$
and $G'$.
\end{Remark*}
%s4 #&#
\section{An abstract posterior contraction theorem}
\label{Sec-gen}

In this section, we state an abstract posterior contraction theorem for
hierarchical
models, whose proof is given in the \hyperref[appB]{Appendix}. The setting
of this theorem is a general hierarchical model defined as follows
\begin{eqnarray*}
G &\sim&\Pi, \qquad \veceta_1,\ldots,\veceta_m | G \sim
P_{\veceta|G},
\\
\SS_{[n]}^{i} | \veceta_i &\sim& P_{\SS_{[n]}|\veceta_i}
\qquad \mbox{for } i=1,\ldots , m.
\end{eqnarray*}
The detail of conditional distributions
in above specifications
is actually irrelevant. Thus, results in this section may be of general interest
for hierarchical models with $m\times n$ data.

As before ${p_{\SS_{[n]}|G}}$ is marginal density of the
generic $\SS_{[n]}$ which is obtained by integrating
out the generic random vector $\veceta$ (e.g., see Equation \eqref
{Eqn-marginal-SSn}).
We need several key notions. Define the Hausdorff ball as:
\[
B_{d_{\mathcal{H}}}(G_1,\delta) := \bigl\{G \in\Delta^{d}\dvt
d_{\mathcal
{H}}(G_1,G) \leq\delta\bigr\}.
\]
A useful quantity for proving posterior concentration theorems
is the Hellinger information of Hausdorff metric for a given set:

%de1 #&#
\begin{definition}
\label{Def-Hellinger}
Fix $G_0 \in\Gcal^*$.
For a fixed $n$, the sample size of $\SS_{[n]}$,
define the Hellinger information of $d_{\mathcal{H}}$ metric for set
$\Gcal
\subset\Gcal^*$
as a real-valued function
on the positive reals $\Psi_{\Gcal,n}\dvtx \real_+ \rightarrow\real$:
%
%e6 #&#
\begin{equation}
\label{Eqn-H} \Psi_{\Gcal,n}(\delta) := \inf_{G\in\Gcal; d_{\mathcal
{H}}(G_0,G) \geq
\delta/2}
h^2(p_{\SS_{[n]}|G_0}, p_{\SS_{[n]}|G}).
\end{equation}
\end{definition}

We also define $\Phi_{\Gcal,n}\dvtx \real_+ \rightarrow\real$
to be an arbitrary non-negative valued function on the positive reals
such that for any $\delta> 0$,
\[
\sup_{G,G' \in\Gcal; d_{\mathcal{H}}(G,G') \leq\Phi_{\Gcal
,n}(\delta)} h^2({p_{\SS_{[n]}|G}},
p_{\SS_{[n]}|G'}) \leq\Psi _{\Gcal
,n}(\delta)/4.
\]
In both definitions of $\Phi$ and $\Psi$, we suppress the dependence
on (the fixed) $G_0$
to simplify notations.
Note that if $G_0 \in\Gcal$, it follows from
the definition that $\Phi_{\Gcal,n}(\delta) < \delta/2$.
%
%pa4.subsection.subsubsection.1 #&#
\begin{Remark*}\label{rem}
Suppose that conditions of Lemma~\ref{Lem-W}(b) hold, so that
\begin{eqnarray*}
h^2(p_{\SS_{[n]}|G},p_{\SS_{[n]}|G'}) \leq K(p_{\SS_{[n]}|G},p_{\SS
_{[n]}|G'})
\leq\frac{n}{c_0}C_0d_{\mathcal{H}}\bigl(G,G'
\bigr).
\end{eqnarray*}
Then it suffices to choose $\Phi_{\Gcal,n}(\delta) = \frac{c_0}{4 n
C_0}\Psi_{\Gcal,n}(\delta)$.
\end{Remark*}

Define the neighborhood of the prior support around $G_0$ in terms of
Kullback--Leibler
distance of the marginal densities ${p_{\SS_{[n]}|G}}$:
%
%e7 #&#
\begin{equation}
\label{Eqn-kl-neighborhood} B_K(G_0,\delta) = \bigl\{G \in
\mathcal{G}^*| K({p_{\SS_{[n]}
|G_0}},{p_{\SS_{[n]}|G}})
\leq\delta^2; K_2({p_{\SS_{[n]}|G_0}},
{p_{\SS_{[n]}|G}}) \leq\delta ^2 \bigr\}.
\end{equation}

%th4 #&#
\begin{theorem}
\label{Thm-Gen}
Let $\Gcal$ denote the support of the prior $\Pi$. Fix $G_0 \in\Gcal$
and suppose that
\begin{enumerate}[(b)]
\item[(a)] $m \rightarrow\infty$ and $n\rightarrow\infty$ at a certain
rate relative to $m$.
\item[(b)] There is a large constant $C$,
a sequence of scalars $\epsilon_{m,n}\rightarrow0$ defined in terms
of $m$
and $n$ such that $m\epsilon_{m,n}^2$ tends to infinity,
such that
%
%e8 #&#
%e9 #&#
\begin{eqnarray}
\label{Eqn-entropy-2} && \sup_{G_1\in\Gcal} \log D\bigl(
\Phi_{\Gcal,n}(\epsilon), \Gcal\cap B_{d_{\mathcal{H}}}(G_1,
\epsilon/2), d_{\mathcal{H}}\bigr)\nonumber
\\
&&\hphantom{\sup_{G_1\in\Gcal}}{} + \log D\bigl(\epsilon/2, \Gcal\cap B_{d_{\mathcal{H}}}(G_0, 2
\epsilon) \setminus B_{d_{\mathcal{H}}}(G_0,\epsilon), d_{\mathcal{H}}
\bigr) \leq m\epsilon_{m,n}^2
\\
&&\quad  \forall\epsilon\geq\epsilon_{m,n},
\nonumber
\\
%&& \Pi(\Gcal\setminus\Gcal_m) \leq\exp[-m\epsmn^2(C+4)], \\
\label{Eqn-support-3} && \Pi\bigl(B_K(G_0,
\epsilon_{m,n})\bigr) \geq\exp\bigl[-m\epsilon_{m,n}^2C
\bigr].  %
\end{eqnarray}
\item[(c)] There is a sequence of positive scalars $M_m$ such that
%
%e10 #&#
%e11 #&#
\begin{eqnarray}
\label{Eqn-support-2} & \Psi_{\Gcal,n}(M_m\epsilon_{m,n})
\geq8\epsilon_{m,n}^2(C+4),&
\\
\label{Eqn-rate-cond} & \exp\bigl(2m\epsilon_{m,n}^2\bigr)\displaystyle\sum
_{j\geq M_m} \exp\bigl[-m\Psi_{\Gcal
,n}(j
\epsilon_{m,n})/8\bigr] \rightarrow0.&
\end{eqnarray}
\end{enumerate}

Then, $\Pi(G\dvt  d_{\mathcal{H}}(G_0,G) \geq M_m\epsilon_{m,n}|\SS_{[n]}^{[m]})
\rightarrow0$ in $P_{\SS_{[n]}|G_0}^m$-probability as $m$ and
$n\rightarrow\infty$.\vadjust{\goodbreak}
\end{theorem}

%The proof of this theorem follows the method of Ghosal, Ghosh and
%van der Vaart \cite{Ghosal-Ghosh-vanderVaart-00} (see also
Condition \eqref{Eqn-entropy-2} is referred to as entropy condition
for certain sets in the support of the prior. Condition \eqref{Eqn-support-3}
is concerned with the ``thickness'' of the prior as measured by
the Kullback--Leibler distance (see also \cite{Ghosal-Ghosh-vanderVaart-00}).
Conditions \eqref{Eqn-support-2}
and \eqref{Eqn-rate-cond} are related to the Hellinger information
function (see also \cite{Nguyen-11}). The proof of this theorem is
deferred to the \hyperref[appB]{Appendix}. As noted above, this result is applicable
to any hierarchical models for $m\times n$ data. The choice of Hausdorff
metric $d_{\mathcal{H}}$ is arbitrary here, and can be replaced by any
other valid
metric (e.g., $d_{\mathcal{M}}$).
The remainder of the paper is devoted to verifying the conditions of this
theorem so it can be applied. These conditions hinge
on our having established a lower bound for the Hellinger
information function $\Psi_{\Gcal,n}(\cdot)$ (via Theorem~\ref{Thm-C}),
and a lower bound for the prior probability defined on Kullback--Leibler balls
$B_K(G_0,\cdot)$
(via Theorem~\ref{Thm-KL}). Both types of results
are obtained by utilizing the convex geometry lemmas
described in the previous section.

%s5 #&#
\section{Inequalities for the Hausdorff distance}
\label{Sec-contraction}

The following results guarantee that as marginal densities of $\SS
_{[n]}$ get
closer in total variation distance metric (or Hellinger metric), so do
the corresponding
population polytopes in Hausdorff metric (or minimum matching metric).
This gives a lower bound for the Hellinger information defined by Equation
\eqref{Eqn-H},
because $h$ is related to $V$ via inequality $h \geq V$.

%th5 #&#
\begin{theorem}
\label{Thm-C}
\begin{enumerate}[(b)]
\item[(a)] Let $G,G'$ be two convex bodies in $\Delta^d$.
$G$ is a $p$-dimensional body containing spherical ball
$B_{p}(\vectheta_c,r)$,
while
$G'$ is $p'$-dimensional body containing $B_{p'}(\vectheta_c,r)$
for some $p,p' \leq d, r> 0, \vectheta_c \in\Delta^d$.
In addition, assume that both $p_{\veceta|G}$ and $p_{\veceta|G'}$
are $\alpha
$-regular densities
on $G$ and $G'$, respectively. Then, there is $c_1>0$ independent
of $G,G'$ such that
\[
c_1 d_{\mathcal{H}}\bigl(G,G'\bigr)^{(p \vee p') + \alpha}
\leq V( {p_{\SS_{[n]}
|G}},{p_{\SS_{[n]}|G'}}) + 6(d+1)
\exp \biggl[-\frac{n}{8(d+1)}d_{\mathcal{H}}\bigl(G,G'
\bigr)^2 \biggr].
\]

\item[(b)] Assume further that $G$ is fixed convex polytope, $G'$ an arbitrary
polytope, $p'= p$, and that either
$|\extr G'| = |\extr G|$
or the pairwise distances of extreme points of $G'$ is bounded
from below by a constant $r_0 > 0$. Then, there are constants $c_2,C_3>0$
depending only on $G$ and $r_0$ (and independent of $G'$) such that
\[
c_2 d_{\mathcal{H}}\bigl(G,G'\bigr)^{1+\alpha}
\leq V({p_{\SS
_{[n]}|G}}, {p_{\SS_{[n]}|G'}}) + 6(d+1)
\exp \biggl[-\frac{n}{C_3(d+1)}d_{\mathcal{H}}\bigl(G,G'
\bigr)^2 \biggr].
\]
\end{enumerate}
\end{theorem}

%pa5.subsection.subsubsection.1 #&#
\begin{Remark*} Part (a) holds for varying pairs of $G,G'$ satisfying
certain conditions. It is consequence of Lemma~\ref{Lem-diff-1}(a).
Part (b)
produces a tighter bound, but it
holds only for a fixed $G$, while $G'$ is allowed to vary
while satisfying certain conditions. This is a consequence of Lemma~\ref{Lem-diff-3}.
Constants $c_1, c_2$ are the same as those from Lemmas \ref
{Lem-diff-1}(a)
and \ref{Lem-diff-3}, respectively.
\end{Remark*}
\begin{pf*}{Proof of Theorem~\ref{Thm-C}}
(a) The main idea of the proof is the construction of a suitable test set
in order to distinguish ${p_{\SS_{[n]}|G'}}$ from
$
{p_{\SS_{[n]}|G}}$. The proof
is organized as a sequence of steps.
%pa5.subsection.subsubsection.2 #&#

\noindent\textit{Step 1}.
Given a data vector $\SS_{[n]}= (X_1,\ldots,X_n)$, define $\hatveceta
(\SS
) \in\Delta^d$ such that
the $i$-element of $\hatveceta(\SS)$ is $\frac{1}{n}\sum_{j=1}^{n}\indicator(X_j=i)$
for each $i=0,\ldots,d$. In the following, we simply use $\hatveceta$
to ease the notations.
By the definition of the variational distance,
%
%e12 #&#
\begin{equation}
\label{Eqn-V} V({p_{\SS_{[n]}|G}},{p_{\SS_{[n]}|G'}})
= \sup_{A} \bigl |P_{\SS_{[n]}|G}(\hatveceta\in A) -
P_{\SS_{[n]}|G'}(\hatveceta\in A)\bigr |,
\end{equation}
where the supremum is taken over all measurable subsets of $\Delta^d$.

\noindent\textit{Step 2}. Fix a constant $\epsilon> 0$. By Hoeffding's
inequality and the union bound,
under the conditional distribution $P_{\SS_{[n]}|\veceta}$,
\[
P_{\SS_{[n]}|\veceta}\Bigl(\max_{i=0,\ldots,d}|\hat{\eta}_i -
\eta_i| \geq \epsilon\Bigr) \leq2(d+1)\exp\bigl(-2n
\epsilon^2\bigr)
\]
with probability one (as $\veceta$ is random).
It follows that
\begin{eqnarray*}
P_{\veceta\times\SS_{[n]}|G}\bigl(\|\hatveceta- \veceta\| \geq\epsilon\bigr) & \leq&
P_{\veceta\times\SS_{[n]}|G}\Bigl(\max_{i=0,\ldots,d}|\hat{\eta}_i -
\eta_i| \geq\epsilon (d+1)^{-1/2}\Bigr)
\\
& \leq& 2(d+1)\exp\bigl[-2n\epsilon^2/(d+1)\bigr].
\end{eqnarray*}
The same bound holds under $P_{\veceta\times\SS_{[n]}|G'}$.

\noindent\textit{Step 3}. Define event $B = \{\|\hatveceta- \veceta\| <
\epsilon\}$.
Take any (measurable) set $A \subset\Delta^d$,
%
%e13 #&#
\begin{eqnarray}\label{Eqn-Vb}
&&\bigl  |P_{\SS_{[n]}|G}(\hatveceta\in A) - P_{\SS_{[n]}|G'}(\hatveceta \in A)\bigr |
\nonumber
\\
&&\quad  =  \bigl |P_{\veceta\times\SS_{[n]}|G}(\hatveceta\in A ; B) + P_{\veceta\times\SS_{[n]}|G}\bigl(
\hatveceta\in A ; B^C\bigr)
\nonumber
\\
&&\hphantom{\quad  =  |}{} - P_{\veceta\times\SS_{[n]}|G'}(\hatveceta\in A ; B) - P_{\veceta\times\SS_{[n]}|G'}\bigl(\hatveceta
\in A ; B^C\bigr)\bigr |
\\
& &\quad \geq \bigl |P_{\veceta\times\SS_{[n]}|G}(\hatveceta\in A ; B) - P_{\veceta
\times\SS_{[n]}|G'}(\hatveceta
\in A ; B)\bigr |
\nonumber
\\
&&\qquad {} - 4(d+1)\exp\bigl[-2n\epsilon^2/(d+1)\bigr]. \nonumber
\end{eqnarray}

\noindent\textit{Step 4}. Let $\epsilon_1 = d_{\mathcal{H}}(G,G')/4$. For
any $\epsilon
\leq\epsilon_1$, recall
the outer $\epsilon$-parallel set $G_{\epsilon} = (G + B_{d+1}(\vec
{0},\epsilon))$,
which is full-dimensional ($d+1$) even though $G$ may not be.
By triangular inequality, $d_{\mathcal{H}}(G_{\epsilon},G_{\epsilon
}') \geq d_{\mathcal{H}}
(G,G')/2$.
We shall argue that for any $\epsilon\leq\epsilon_1$,
there is a constant $c_1 > 0$ independent of $G,G'$, $\epsilon$ and
$\epsilon_1$ such that either one of the two scenarios
holds:
\begin{enumerate}[(ii)]
\item[(i)] There is a set $A^* \subset G \setminus G'$ such that
$A^*_{\epsilon} \cap G_\epsilon'= \emptyset$ and $\vol_p(A^*) \geq
c_1 \epsilon_1^p$, or
\item[(ii)] There is a set $A^* \subset G' \setminus G$ such that
$A^*_\epsilon\cap G_\epsilon= \emptyset$ and $\vol_{p'}(A^*) \geq
c_1 \epsilon_1^{p'}$.
\end{enumerate}

Indeed, since $\epsilon\leq d_{\mathcal{H}}(G,G')/4$, either
one of the following two inequalities holds:
$d_{\mathcal{H}}(G \setminus G_{3\epsilon}', G') \geq d_{\mathcal
{H}}(G,G')/4$ or $d_{\mathcal{H}}(G'
\setminus G_{3\epsilon},
G) \geq d_{\mathcal{H}}(G,G') /4$. If the former inequality
holds, let $A^* = G\setminus G_{3\epsilon}'$.
Then, $A^* \subset G\setminus G'$ and $A^*_\epsilon\cap G_\epsilon'=
\emptyset$.
Moreover, by Lemma~\ref{Lem-diff-1}(a), $\vol_p(A^*) \geq
c_1\epsilon_1^p$, for
some constant $c_1 > 0$ independent of $\epsilon, \epsilon_1, G, G'$, so
$A^*$ satisfies (i).
%Note also that
%all elements in set $A^*$ is within distance $\epsilon_1$ from the
%boundary of $G$.
In fact, using the same argument as in the proof of
Lemma~\ref{Lem-diff-1}(a) there is a point $x \in\bd G$
such that $G' \cap B_{p}(x,\epsilon_1) = \emptyset$.
Combined with the $\alpha$-regularity of $P_{\veceta|G}$, we have
$P_{\veceta|G}(A^*) \geq\epsilon^\alpha\vol_p(G\cap
B_{p}(x,\epsilon_1))
\geq c_1\epsilon^{p+\alpha}$ for some constant $c_1>0$.
If the latter inequality holds, the same argument
applies by defining $A^* = G' \setminus G_{3\epsilon}$ so that (ii) holds.

\noindent\textit{Step 5}.
Suppose that (i) holds for the chosen $A^*$. This means that
$P_{\veceta\times\SS_{[n]}|G'}(\hatveceta\in A^*_\epsilon; B) \leq
P_{\veceta|G'}(\veceta\in
A_{2\epsilon}^*) = 0$,
since $A_{2\epsilon}^* \cap G' = \emptyset$, which is a consequence
of $A_{\epsilon}^* \cap G_\epsilon'= \emptyset$.
In addition,
\begin{eqnarray*}
P_{\veceta\times\SS_{[n]}|G}\bigl(\hatveceta\in A^*_\epsilon; B\bigr) & \geq &
P_{\veceta\times\SS_{[n]}|G}\bigl(\veceta \in A^*; B\bigr)
\\
& \geq& P_{\veceta|G}\bigl(A^*\bigr) - P_{\veceta\times\SS_{[n]}|G}\bigl(B^C
\bigr)
\\
& \geq& P_{\veceta|G}\bigl(A^*\bigr) - 2(d+1)\exp\bigl(-2n
\epsilon^2/(d+1)\bigr)
\\
& \geq& c_1 \epsilon_1^{p+\alpha} - 2(d+1)\exp
\bigl(-2n\epsilon^2/(d+1)\bigr).
\end{eqnarray*}
Hence, by Equation \eqref{Eqn-Vb}
$|P_{\SS_{[n]}|G}(\hatveceta\in A^*_\epsilon) - P_{\SS
_{[n]}|G'}(\hatveceta\in
A^*_\epsilon)|
\geq c_1 \epsilon_1^{p+\alpha} - 6(d+1)\times\break \exp(-2n\epsilon^2)$. Set
$\epsilon=
\epsilon_1$, the conclusion then
follows by invoking Equation \eqref{Eqn-V}. The scenario of (ii) proceeds
in the
same way.

(b) Under the condition that the pairwise distances of extreme points of
$G'$ are bounded from below by $r_0>0$, the proof is very similar
to part (a), by invoking Lemma~\ref{Lem-diff-3}. Under the condition
that $|\extr G'| = k$, the proof is also similar, but it requires
a suitable modification for the existence of set $A^*$.
For any small $\epsilon$, let $\tilde{G}_\epsilon$ be the
minimum-volume homethetic
transformation of $G$, with respect to center $\vectheta_c$, such that
$\tilde{G}_\epsilon$ contains
$G_\epsilon$. Since $B_p(\vectheta_c,r) \subset G \subset
B_p(\vectheta_c,R)$ for
$R=1$, it is simple to see that $d_{\mathcal{H}}(G,\tilde{G}_\epsilon)
\leq\epsilon R/r = \epsilon/r$.
\comment{For any small $\epsilon$, let $\tilde{G}_\epsilon$ be
a convex polytope obtained by the intersection of all $G$-containing
half-spaces supported by the hyperplanes which are parallel to and
of distance $\epsilon$ away from each of $G$'s $p-1$-dimensional faces.
It is simple to see that $\tilde{G}_\epsilon$ has the same number
of extreme points as $G$, and $G_\epsilon\subset\tilde{G}_\epsilon$.
Since $G$ contains a solid $p$-dimensional ball of radius $r>0$,
both $G$ and $\tilde{G}_\epsilon$ do not share sharp corners, so
there exists a constant $C_2 > 0$ independent of $G$ and $\epsilon$
such that the distances of all pairs of corresponding vertices of $G$
and $G'$
are bounded by $C_2\epsilon$. It follows that, by Lemma~\ref{Lem-M},
$d_{\mathcal{H}}(G,\tilde{G}_\epsilon) \leq
C_2\epsilon$. Likewise let $\tilde{G}_\epsilon'$ be associated
with $G'$ in the same manner.}

Set $\epsilon_1 = d_{\mathcal{H}}(G,G')r/4$.
We shall argue that
for any $\epsilon\leq\epsilon_1$, there is a constant $c_0 > 0$
independent of
$G'$, $\epsilon$ and $\epsilon_1$ such that either one of the following
two scenarios hold:
\begin{enumerate}[(iii)]
\item[(iii)] There is a set $A^* \subset G \setminus G'$ such that
$A^*_{\epsilon} \cap G_\epsilon'= \emptyset$ and $\vol_p(A^*) \geq
c_2 \epsilon_1$, or
\item[(iv)] There is a set $A^* \subset G' \setminus G$ such that
$A^*_\epsilon\cap G_\epsilon= \emptyset$ and $\vol_{p}(A^*) \geq
c_2 \epsilon_1$.
\end{enumerate}
Indeed, note that either one of the following two inequalities
holds: $d_{\mathcal{H}}(G\setminus\tilde{G}_{3\epsilon}', G')
\geq d_{\mathcal{H}}(G,G')/4$ or $d_{\mathcal{H}}(G'\setminus\tilde
{G}_{3\epsilon}, G)
\geq d_{\mathcal{H}}(G,G')/4$. If the former inequality holds, let
$A^* = G\setminus\tilde{G}_{3\epsilon}'$. Then,
$A^* \subset G\setminus G'$ and $A^*_\epsilon\cap\tilde
{G}_{\epsilon}' = \emptyset$.
Observe that both $G$ and $\tilde{G}_{3\epsilon}'$ have
the same number of extreme points by the construction. Moreover,
$G$ is fixed so that all geometric Properties \ref{propA2}, \ref{propA1} are satisfied
for both
$G$ and $\tilde{G}_{3\epsilon}'$ for sufficiently
small $d_{\mathcal{H}}(G,G')$. By Lemma~\ref{Lem-diff-3}, $\vol_p(A^*)
\geq c_2 \epsilon_1$. Hence, (iii) holds.
If the latter inequality holds, the same argument applies by
defining $A^* = G' \setminus\tilde{G}_{3\epsilon}$ so that (iv) holds.

Now the proof of the theorem proceeds in the same manner as in part (a).
\end{pf*}

%s6 #&#
\section{Concentration properties of the prior support}
\label{Sec-concentration}

In this section, we study properties of the support
of the prior probabilities as specified by the admixture model,
including bounds for the support of the prior as defined by
Kullback--Leibler neighborhoods.

%pa6.subsection.subsubsection.1 #&#
\subsection*{\texorpdfstring{$\alpha$}{alpha}-regularity}
Let $\vecbeta$ be a random variable taking values in $\Delta^{k-1}$
that has
a density $p_{\vecbeta}$ (with respect to the $k-1$-dimensional
Hausdorff measure
$\Hcal^{k-1}$ on $\real^k$). For a definition of the Hausdorff
measure, see \cite{Evans-Gariepy-92}, which
in our case reduces to the Lebesgue measure defined on simplex $\Delta^{k-1}$.
Define random variable $\veceta= \beta_1\vectheta_1 + \cdots+ \beta
_k \vectheta_k$, which
takes values in $G = \conv(\vectheta_1,\ldots,\vectheta_k)$.
Write $\veceta= L \vecbeta$, where $L = [\vectheta_1\quad  \cdots\quad
\vectheta_{k}]$
is a $(d+1)\times k$ matrix.
%Suppose that $P(\cdot|\gamma)$ places a positive density on $
%$k-1$-dimensional Hausdorff measure on $\real^{k}$ (which concides
%with the $k-1$-dimensional
%Lebesgue measure $\Lcal$ on $\real^{k-1}$ for the first $k-1$
%coordinates of $\vec{\beta}$.)
%
If $k \leq d+1$, $\vectheta_1,\ldots,\vectheta_k$ are generally
linearly independent,
in which case matrix $L$ has rank $k-1$.
By the change of variable formula \cite{Evans-Gariepy-92} (Chapter~3),
$P_{\vecbeta}$
induces a distribution $P_{\veceta|G}$ on $G \subset\Delta^{d}$, which
admits the following
density with respect to the $k-1$ dimensional Hausdorff measure $\Hcal^{k-1}$
on $\Delta^{d}$:
%
%e14 #&#
\begin{equation}
\label{Eqn-density-1} p_{\veceta}(\veceta|G) = p_{\vecbeta}
\bigl(L^{-1}(\veceta)\bigr) J(L)^{-1}.
\end{equation}
Here, $J(L)$ denotes the Jacobian of the linear map.
%Note that when $k-1 \leq d$, $L$ can be expressed as
%$L = O_{(d+1)_\times k} S_{k\times k}$, where
%$S$ is symmetric and full-ranked, and $O$ consists of $k$ orthonormal
%column
%vectors in $\real^{d+1}$.
%Then $L^{-1} = S^{-1}O^T$, and $J(L) = \det(S)$.
%
On the other hand, if $k \geq d+1$, then $L$ is generally $d$-ranked.
%Delete the bottom row of $L$ to obtain
%$d\times k$ matrix $\tL$. Write
%$\tL= S_{d\times d} O^T$, where $S$ is symmetric and invertible, and
%$O_{k\times d}$ consists of $d$ orthonormal column vectors in $
%Then $J(\tL) = \det(S)$.
The induced distribution for $\veceta$ admits
the following density with respect to the $k-(d+1)$-dimensional Hausdorff
measure on $\real^{d+1}$:
%
%e15 #&#
\begin{equation}
\label{Eqn-density-2} p_{\veceta}(\veceta|G) = \int_{L^{-1}\{\veceta\}}
p_{\vecbeta}(\vecbeta) J(L)^{-1} \Hcal ^{k-(d+1)}(\mathrm{d}\vecbeta).
\end{equation}
\comment{
[[re-write]] Here $J(L)$ denotes the Jacobian of the linear map $L:
\Delta^{k-1}
\rightarrow\Delta^{d}$, $\vecbeta\mapsto\veceta= L\vecbeta$, where
in the last equation, we also use $L$ to denote the $(d+1)\times k$ matrix,
$L = [\vectheta_1 \ldots\vectheta_k]$. When $k \geq d+1$, $L$
generally has rank $d$ (
as assumed).
Let $L^\dagger$ be the pseudo-inverse $k \times(d+1)$ matrix of $L$.
For any $\veceta\in G = \conv(\vectheta_1,\ldots, \vectheta_k)$,
$L^{-1}(\veceta) = \Delta^{k-1} \cap
\{ L^\dagger\veceta+ (I-L^\dagger L) z | z \in\real^k\}$.
Note that $(I-L^\dagger L)$ represents the orthogonal projection from
$\real^k$
onto the null space $\mathcal{N}(L)$, which has $k-d$ dimensions,
so $L^{-1}(\veceta)$ has $k-d-1$ dimensions.
}

A common choice for $P_{\vecbeta}$ is the Dirichlet distribution, as
adopted by \cite{Pritchard-etal-00,Blei-etal-03}:
given parameter $\gamma\in\real_{+}^{k}$, for any $A \subset\Delta^{k-1}$,
\[
P_{\vecbeta} (\vec{\beta} \in A|\gamma) = \int_{A}
\frac{\Gamma(\sum\gamma_j)}{ \prod_{j=1}^{k}\Gamma(\gamma_j)} \prod_{j=1}^{k}
\beta_j^{\gamma_j-1} \Hcal^{k-1}(\mathrm{d}\vecbeta).
\]

%le5 #&#
\begin{lemma}
\label{Lem-Dirichlet}
Let $\veceta= \sum_{j=1}^{k} \beta_j \vectheta_j$, where $\vecbeta
$ is
distributed according to a $k-1$-dimensional Dirichlet distribution
with parameters $\gamma_j \in(0, 1]$ for $j=1,\ldots,k$.
\begin{enumerate}[(a)]
\item[(a)] If $k\leq d+1$, there is constant $\epsilon_0 = \epsilon
_0(k) >0$,
and constant $c_6 = c_6(\gamma,k,d) > 0$ dependent on $\gamma, k$ and
$d$ such that for
any $\epsilon< \epsilon_0$,
\[
\inf_{G \subset\Delta^d}\inf_{\veceta^* \in G} P_{\veceta|G}
\bigl(\bigl \|\veceta- \veceta^*\bigr \| \leq\epsilon\bigr) \geq c_6
\epsilon^{k-1}.
\]
\item[(b)] If $k>d+1$, the statement holds with a lower bound
$c_6 \epsilon^{d+\sum_{i=1}^{k}\gamma_i}$.
\end{enumerate}
\end{lemma}

A consequence of this lemma is that if $\gamma_j \leq1$ for all
$j=1,\ldots,k$,
$k\leq d+1$ and $G$ is $k-1$-dimensional, then
the induced $P_{\veceta|G}$ has a Hausdorff density that is bounded
away from
0 on
the entire its support $\Delta^{k-1}$, which implies $0$-regularity.
On the other hand, if $\gamma_j \leq1$ for all $j$,
$k > d+1$, and $G$ is $d$-dimensional, the $P_{\veceta|G}$ is at least
$\sum_{j=1}^{k}\gamma_j$-regularity. Note that the $\alpha$-regularity condition is concerned with
the density behavior near the boundary of its support, and thus is weaker
than what is guaranteed here.

%pa6.subsection.subsubsection.2 #&#
\subsection*{Bounds on KL divergences}

Suppose that the population polytope $G$ is endowed with a prior
distribution on $\Gcal^k$ (via prior on the population structures
$\vectheta_1,\ldots,\vectheta_k$).
Given $G$, the marginal density ${p_{\SS_{[n]}|G}}$ of
$n$-vector $\SS_{[n]}$ is
obtained via Equation \eqref{Eqn-marginal-SSn}.
To establish the concentration properties of Kullback--Leibler
neighborhood $B_K$ as induced by the prior, we need
to obtain an upper bound on the KL divergences for the marginal densities
in terms of Hausdorff metric on population polytopes.
First, consider a very special case.

%le6 #&#
\begin{lemma}
\label{Lem-KL-ez}
Let $G,G' \subset\Delta^d$ be closed convex sets satisfying Property
\textup{\ref{propA1}}.
Moreover, assume that
\begin{enumerate}[(b)]
\item[(a)] $G \subset G'$, $\aff G = \aff G'$ is $p$-dimensional, for
$p \leq d$.
\item[(b)] $P_{\veceta|G}$ (resp. $P_{\veceta|G'}$) are uniform
distributions on $G$ (resp. $G'$).
\end{enumerate}
Then, there is a constant $C_1 = C_1(r,p) >0$ such that
$K({p_{\SS_{[n]}|G}}, {p_{\SS_{[n]}|G'}})\leq
C_1 d_{\mathcal{H}}(G,G')$.
\end{lemma}

\begin{pf}
First, we note a well-known fact of KL divergences: the divergence
between marginal distributions
(e.g., on $\SS_{[n]}$)
is bounded from above by the divergence between joint distributions
(e.g., on $\veceta$
and $\SS_{[n]}$ via Equation \eqref{Eqn-joint-eta-SSn}):
\[
K({p_{\SS_{[n]}|G}},{p_{\SS_{[n]}|G'}}) \leq
K(P_{\veceta\times\SS_{[n]}|G} , P_{\veceta\times\SS_{[n]}|G'}).
\]
Due to the hierarchical specification,
$p_{\veceta\times\SS_{[n]}|G} = p_{\veceta|G} \times p_{\SS
_{[n]}|\veceta}$
and $p_{\veceta\times\SS_{[n]}|G'} = p_{\veceta|G'} \times p_{\SS_{[n]}
|\veceta}$,
so $K(P_{\veceta\times\SS_{[n]}|G}, P_{\veceta\times\SS
_{[n]}|G'}) = K(p_{\veceta|G}, p_{\veceta|G'})$.
The assumption $\aff G = \aff G'$ and moreover $G\subset G'$
implies that $K(p_{\veceta|G}, p_{\veceta|G'}) < \infty$. In addition,
$P_{\veceta|G}$ and $P_{\veceta|G'}$ are assumed to be
uniform distributions on $G$ and $G'$, respectively, so
\[
K(p_{\veceta|G}, p_{\veceta|G'}) = \int\log\frac{1/\vol_{p} G} {
1/\vol_{p} G'}
\,\mathrm{d}P_{\veceta|G}.
\]
By Lemma~\ref{Lem-diff-1}(b),
$\log[\vol_{p} G'/\vol_{p}G] \leq\log(1+C_1 d_{\mathcal
{H}}(G,G')) \leq C_1
d_{\mathcal{H}}(G,G')$
for some constant $C_1 = C_1(r,p) > 0$. This completes the proof.
\end{pf}
%
%pa6.subsection.subsubsection.3 #&#
\begin{Remark*}
The previous lemma requires a particularly stringent condition, $\aff G
= \aff G'$,
and moreover $G\subset G'$,
which is usually violated when $k < d+1$. However,
the conclusion is worth noting in that the upper bound does not depend
on the
sample size $n$ (for $\SS_{[n]}$).
The next lemma removes this condition and the condition that
both $p_{\veceta|G}$ and $p_{\veceta|G'}$ be uniform. As a result the
upper bound
obtained is weaker, in the sense that the bound is not
in terms of a Hausdorff distance, but
in terms of a Wasserstein distance.
\end{Remark*}

Let $Q(\veceta_1,\veceta_2)$ denote a coupling of $P(\veceta|G)$ and
$P(\veceta|G')$,
that is, a joint distribution on $G \times G'$
whose induced marginal distributions of $\veceta_1$ and $\veceta_2$ are
equal to $P(\veceta|G)$ and $P(\veceta|G')$, respectively. Let $\Qcal
$ be
the set of all such couplings. The Wasserstein distance between
$p_{\veceta|G}
$ and $p_{\veceta|G'}$ is
defined as
\[
\W_1(p_{\veceta|G},p_{\veceta|G'}) = \inf
_{Q \in\Qcal} \int\| \veceta _1-\veceta_2\|
\,\mathrm{d}Q(\veceta_1,\veceta_2).
\]

%le7 #&#
\begin{lemma}
\label{Lem-KL-bound}
Let $G,G' \subset\Delta^d$ be closed convex subsets such that
any $\veceta= (\eta_0,\ldots,\eta_d)
\in G \cup G'$ satisfies $\min_{l=0,\ldots,d}\eta_l > c_0$ for
some constant $c_0>0$. Then\vspace*{-1pt}
\[
K({p_{\SS_{[n]}|G}},{p_{\SS_{[n]}|G'}}) \leq
\frac
{n}{c_0} \W_1(p_{\veceta|G},p_{\veceta|G'}).
\]
%
%e^{n \Diam(G \cup G')/c_0}$.
\end{lemma}

%pa6.subsection.subsubsection.4 #&#
\begin{Remark*} As $n \rightarrow\infty$, the upper bound tends
to infinity.
This is expected, because the marginal distribution $P_{\SS_{[n]}|G}$
should degenerate.
Since typically $\aff G \neq\aff G'$, Kullback--Leibler distances
between $P_{\SS_{[n]}|G}$ and
$P_{\SS_{[n]}|G'}$ should typically tend to infinity.
\end{Remark*}
\begin{pf*}{Proof of Lemma~\ref{Thm-KL}}
Associating each sample $\SS_{[n]}= (X_1,\ldots, X_n)$
with a $d+1$-dimensional vector $\veceta(\SS) \in\Delta^{d}$,
where $\veceta(\SS)_i = \frac{1}{n}\sum_{j=1}^{n}\indicator(X_j =
i)$ for each $i=0,\ldots,d$.
The density of $\SS_{[n]}$ given $G$ (with respect to the counting measure)
takes the form:\vspace*{-1pt}
\[
{p_{\SS_{[n]}|G}}(\SS_{[n]}) = \int_{G}
p(\SS_{[n]}|\veceta) \,\mathrm{d}P(\veceta|G) = \int_{G} \exp
\Biggl( n \sum_{i=0}^{d} \veceta(
\SS)_i \log\veceta_i \Biggr) \,\mathrm{d} P(\veceta|G) .
\]

\comment{
$\veceta(\SS)$ is alternatively viewed as a probability distribution
on $\{0,\ldots,d\}$. Let
$H(\veceta(\SS)) = -\sum_{i} \veceta(\SS)_i \log\veceta(\SS)_i$
be the entropy functional
defined on $\veceta(\SS)$.
Note that for any $\veceta\in G$,
\[
\sum_{i=0}^{d} \veceta(\SS)_i
\log\veceta_i = -H\bigl(\veceta(\SS )\bigr)-\sum
_{i} \veceta(\SS)_i \log\bigl(\veceta(
\SS)_i/\veceta_i\bigr) = -H\bigl(\veceta(\SS)\bigr) - D
\bigl(\veceta(\SS)||\veceta\bigr).
\]
}

Due to the convexity of Kullback--Leibler divergence, by Jensen inequality,
for any coupling $Q\in\Qcal$:\vspace*{-1pt}
\begin{eqnarray*}
K({p_{\SS_{[n]}|G}},{p_{\SS_{[n]}|G'}})& =& K \biggl(\int
p(\SS_{[n]}|\veceta_1) \,\mathrm{d}Q(\veceta_1,
\veceta_2), \int p(\SS_{[n]}|\veceta_2) \,\mathrm{d}Q(
\veceta_1,\veceta_2) \biggr)
\\
&\leq&\int K\bigl(p(\SS_{[n]}|\veceta_1),p(
\SS_{[n]}|\veceta_2)\bigr)\, \mathrm{d}Q(\veceta _1,
\veceta_2).
\end{eqnarray*}
It follows that
$K({p_{\SS_{[n]}|G}},{p_{\SS_{[n]}|G'}}) \leq
\inf_{Q}
\int K(p_{\SS_{[n]}|\veceta_1},p_{\SS_{[n]}|\veceta_2})
\,\mathrm{d}Q(\veceta_1,\veceta_2)$.

Note that
$K(P_{\SS_{[n]}|\veceta_1},P_{\SS_{[n]}|\veceta_2}) = \sum_{\SS_{[n]}}
n(K(\veceta(\SS),\veceta_2) - K(\veceta(\SS),\veceta_1))
p_{\SS_{[n]}|\veceta_1}$, where the summation is taken over all
realizations of $\SS_{[n]}\in\{0,\ldots,d\}^n$.
For any $\veceta(\SS) \in\Delta^d$, $\veceta_1 \in G$ and $\veceta
_2 \in G'$,\vspace*{-1pt}
\begin{eqnarray*}
\bigl |K\bigl(\veceta(\SS),\veceta_1\bigr) - K\bigl(\veceta(\SS),
\veceta_2\bigr)\bigr | & = & \Biggl|\sum_{i=0}^{d}
\veceta(\SS)_i \log(\eta_{1,i}/\eta_{2,i})\Biggr|
\\
& \leq& \sum_{i} \veceta(\SS)_i |
\eta_{1,i}-\eta_{2,i}|/c_0
\\
& \leq& \biggl(\sum_{i} \veceta(
\SS)_i^2\biggr)^{1/2} \|\veceta_1 -
\veceta _2\|/c_0
\\
& \leq& \|\veceta_1 - \veceta_2\|/c_0.
\end{eqnarray*}
Here, the first inequality is due the assumption, the second due to
Cauchy--Schwarz.
It follows that
$K(P_{\SS_{[n]}|\veceta_1},P_{\SS_{[n]}|\veceta_2}) \leq n\|\veceta
_1-\veceta_2\|/c_0$,
so $K({p_{\SS_{[n]}|G}},{p_{\SS_{[n]}|G'}})
\leq\frac
{n}{c_0}\W_1(p_{\veceta|G},p_{\veceta|G'})$.
\end{pf*}
\begin{lemma}
\label{Lem-W} Let $G = \conv(\vectheta_1,\ldots, \vectheta_k)$
and $G' = \conv(\vectheta'_1,\ldots, \vectheta'_k)$ (same $k$).
A random variable $\veceta\sim P_{\veceta|G}$ is parameterized by
$\veceta= \sum_{j} \beta_j \veceta_j$,
while a random variable $\veceta\sim P_{\veceta|G'}$
is parameterized by $\veceta= \sum_{j} \beta'_j \veceta'_j$, where
$\vecbeta$ and $\vecbeta'$ are both distributed according to a
symmetric probability
density $p_\vecbeta$.
\begin{enumerate}[(a)]
\item[(a)] Assume that both $G,G'$ satisfy Property \textup{\ref{propA2}}.
Then, for small $d_{\mathcal{H}}(G,G')$,
$\W_1(p_{\veceta|G},\allowbreak   p_{\veceta|G'}) \leq C_0 d_{\mathcal
{H}}(G,G')$ for some constant
$C_0$ specified by Lemma~\ref{Lem-M}.

\item[(b)] Assume further that assumptions in Lemma~\ref{Lem-KL-bound} hold, then
$K({p_{\SS_{[n]}|G}},{p_{\SS_{[n]}|G'}}) \leq
\frac
{n}{c_0}C_0\* d_{\mathcal{H}}(G,G')$.
\end{enumerate}
\end{lemma}

%pa6.subsection.subsubsection.5 #&#
\begin{Remark*} In order to obtain an upper bound for
$K({p_{\SS_{[n]}|G}},{p_{\SS_{[n]}|G'}})$
in terms of $d_{\mathcal{H}}(G,G')$, the assumption that $p_{\vecbeta
}$ is symmetric
appears essential. That is, random variables $\beta_1,\ldots, \beta_k$
are exchangeable under $p_{\vecbeta}$.
Without this assumption, it is possible to have
$d_{\mathcal{H}}(G,G') = 0$, but $K({p_{\SS
_{[n]}|G}},{p_{\SS_{[n]}
|G'}}) > 0$.
\end{Remark*}
\begin{pf*}{Proof of Lemma~\ref{Lem-W}}
By Lemma~\ref{Lem-M} under Property \ref{propA2},
$d_{\mathcal{M}}(G,G') \leq C_0 d_{\mathcal{H}}(G,G')$ for some
constant $C_0$.
Let $d_{\mathcal{H}}(G,G')\leq\epsilon$ for some small $\epsilon> 0$.
Assume without loss of generality that $|\vectheta_j -\vectheta'_j|
\leq C_0
\epsilon$ for all $j = 1,\ldots, k$ (otherwise, simply relabel the subscripts
for $\vectheta_j'$'s).

Let $Q(\veceta,\veceta')$ be a coupling of $P_{\veceta|G}$ and
$P_{\veceta|G'}$
such that under $Q$, $\veceta= \sum_{j=1}^{k} \beta_j \vectheta_j$
and\vspace*{-2pt} $\veceta' = \sum_{j=1}^{k} \beta_j \vectheta'_j$, that is,
$\veceta$
and $\veceta'$ share the \emph{same} $\vecbeta$, where $\vecbeta$
is a random variable\vspace*{1pt} with density $p_{\vecbeta}$.
This is a valid coupling, since $p_{\vecbeta}$ is assumed to
be symmetric.

Under\vspace*{1pt} distribution $Q$, $\E\|\veceta- \veceta'\|
\leq\E\sum_{j=1}^{k}\beta_j \|\vectheta_j - \vectheta'_j \|
\leq C_0 \epsilon\E\sum_{j=1}^{k} \beta_j = C_0 \epsilon$.
Hence, $\W_1(P_{\veceta|G},P_{\veceta|G'}) \leq C_0 \epsilon$.
Part (b) is an immediate consequence.
\end{pf*}

Recall the definition of Kullback--Leibler neighborhood given by
Equation \eqref{Eqn-kl-neighborhood}. We are now ready to
prove the main result of this section.

%th6 #&#
\begin{theorem}
\label{Thm-KL}
Under assumptions \textup{(S1)} and \textup{(S2)}, for any $G_0$ in the support
of prior $\Pi$, for any $\delta> 0$ and $n> \log(1/\delta)$
\[
\Pi\bigl(G \in B_K(G_0,\delta)\bigr) \geq c\bigl(
\delta^2/n^3\bigr)^{kd},
\]
where constant $c = c(c_0,c'_0)$ depends only on $c_0, c'_0$.
\end{theorem}

\begin{pf}
We shall invoke a bound of \cite{Wong-Shen-95} (Theorem~5) on
the KL divergence. This bound says that if $p$ and $q$ are
two densities on a common space such that $\int p^2/q < M$, then for some
universal constant $\epsilon_0 > 0$, as long as $h(p,q) \leq\epsilon
< \epsilon_0$,
there holds: $K(p,q) = \mathrm{O}(\epsilon^2 \log(M/\epsilon))$,
and $K_2(p,q) := \int p(\log(p/q))^2 = \mathrm{O}(\epsilon^2 [\log(M/\epsilon)]^2)$,
where the big O constants are universal.

Let $G_0 = \conv(\vectheta_1^*,\ldots,\vectheta_k^*)$.
Consider a random set $G \in\Gcal^k$ represented by $G = \conv
(\vectheta_1,
\ldots,\allowbreak  \vectheta_k)$, and the event $\mathcal{E}$ that
$\|\vectheta_j - \vectheta_j^*\|\leq
\epsilon$ for all $j=1,\ldots,k$. For the pair of $G_0$ and $G$,
consider a coupling $Q$ for $P_{\veceta|G}$ and $P_{\veceta|G_0}$
such that
any $(\veceta_1,\veceta_2)$ distributed by $Q$ is parameterized
by $\veceta_1 = \beta_1\vectheta_1 + \cdots+\beta_k \vectheta_k$
and $\veceta_2 = \beta_1\vectheta_1^* + \cdots+\beta_k \vectheta_k^*$
(that is, under the coupling $\veceta_1$ and $\veceta_2$
share the same vector $\vecbeta$).
Then, under $Q$, $\E\|\veceta_1 - \veceta_2 \| \leq\epsilon$.
This entails that $\W_1(P_{\veceta|G}, P_{\veceta|G_0}) \leq
\epsilon$. (We note
here that the argument appears similar to the one from
Lemma~\ref{Lem-W}, but we do not need to assume that $p_\vecbeta$
be symmetric in this\vadjust{\goodbreak} theorem.)
If $G$ is randomly distributed according to prior $\Pi$, under
assumption (S2),
the probability of event $\mathcal{E}$
is lower bounded by $c'_0\epsilon^{kd}$.
By Lemma~\ref{Lem-KL-bound},
$h^2(p_{G_0},p_{G}) \leq K(p_{G_0},p_{G})/2 \leq(n/c_0) W_1(P_{\veceta|G}
,P_{\veceta|G_0})
\leq n\epsilon/(2c_0)$.
Note that the density ratio ${p_{\SS_{[n]}|G}}/
{p_{\SS_{[n]}|G_0}}\leq(1/c_0)^n$, which
implies that
$\sum_{\SS_{[n]}} {p_{\SS_{[n]}|G_0}}^2/
{p_{\SS_{[n]}
|G}}\leq(1/c_0)^n$.
We can\vspace*{2pt} apply the
upper bound described in the previous paragraph to obtain:
\[
K_2({p_{\SS_{[n]}|G_0}},{p_{\SS_{[n]}|G}}) =
\mathrm{O} \biggl( \frac{n\epsilon}{2c_0} \biggl[\frac{1}{2}\log\frac{2c_0}{n\epsilon} + n
\log\frac{1}{c_0} \biggr]^2 \biggr).
\]
Here, the big O constant is universal.
If we set $\epsilon= \delta^2/n^3$, then the quantity in the
right hand side of the previous display is bounded by $\mathrm{O}(\delta^2)$
as long as $n > \log(1/\delta)$.
Combining with the probability bound $c'_0\epsilon^{kd}$ derived above,
we obtain the desired result.
\end{pf}

%s7 #&#
\section{Proofs of main theorems and auxiliary lemmas}
\label{Sec-proofs}

%pa7.subsection.subsubsection.1 #&#
\begin{pf*}{Proof of Theorem \protect\ref{Thm-Main} (Overfitted setting)}
The proof proceeds by verifying conditions of Theorem~\ref{Thm-Gen}.
Let $\epsilon_{m,n}= (\log m/m)^{1/2} + (\log n/m)^{1/2}
+ (\log n/n)^{1/2}$.
%
%Choose the sequence of subsets $\Gcal_m$ simply to be the support of
%prior $\Pi$,
%so that $\Pi(\Gstar\setminus\Gcal_m) = 0$. Note that $\Gcal_m
%Condition \eqref{Eqn-support-1} trivially
%holds.
%
Let $\Gcal:= \supp\Pi\subset\Gcal^k$.
Starting with the entropy condition \eqref{Eqn-entropy-2}, we note that
\begin{eqnarray*}
\log D\bigl(\epsilon/2,\Gcal\cap B_{\Hcal}(G_0,2\epsilon),
d_{\mathcal{H}}\bigr) & \leq& \log N\bigl(\epsilon/4, \Gcal\cap
B_{\Hcal}(G_0,2\epsilon), d_{\mathcal{H}} \bigr) = \mathrm{O}(1).
\end{eqnarray*}

By Theorem~\ref{Thm-C}(a), assumption (S3a) and the general inequality that
$h \geq V$, we have:
\[
\Psi_{\Gcal,n}( \epsilon) \geq \bigl[c_1 (
\epsilon/2)^{p+\alpha} - 6(d+1)\mathrm{e}^{-n\epsilon^2/32(d+1)}\bigr]^2,
\]
where $p = \min(k-1,d)$.
So $\Psi_{\Gcal,n}(\epsilon) \geq c\epsilon^{2(p+\alpha)}$ as long as
$c_1(\epsilon/2)^{p+\alpha} \geq12(d+1)\exp[-n\epsilon
^2/\allowbreak  32(d+1)]$. Here, $c$
is a constant depending on $c_1, p, d$. This
is satisfied if $\epsilon$ is bounded from below by a large multiple
of $\epsilon_{m,n}> (\log n/n)^{1/2}$. Using $\Phi_{\Gcal,n}(\delta
) :=
\frac{c_0}{4nC_0}
\Psi_{\Gcal,n}(\delta)$ (cf. \hyperref[rem]{Remark} following Definition~\ref
{Def-Hellinger})
it follows that
\begin{eqnarray*}
&& \log D\bigl(c_0\Psi_{\Gcal,n}(\epsilon)/(4nC_0),
\Gcal\cap B_{\Hcal
}(G_1,\epsilon/2), d_{\mathcal{H}}\bigr)
\\
&&\quad  \leq \log N\bigl(c_0 c\epsilon^{2(p+\alpha)}/(4nC_0),
\Gcal\cap B_{\Hcal}(G_1,\epsilon/2), d_{\mathcal{H}}\bigr)
\\
&&\quad  \lesssim \log\bigl(n^{kd}\epsilon^{-(2p+2\alpha-1)kd}\bigr) \leq m
\epsilon^2,
\end{eqnarray*}
where the last inequality holds since
$\epsilon$ is bounded from below by a large multiple of
$\epsilon_{m,n}> (\log n/m)^{1/2} + (\log m/m)^{1/2}$. Thus, the
entropy condition
\eqref{Eqn-entropy-2} is established.

To verify condition Equation \eqref{Eqn-rate-cond}, we note
that for some constant $c>0$,
\begin{eqnarray*}
&& \exp\bigl(2m\epsilon_{m,n}^2\bigr)\sum
_{j\geq M_m} \exp\bigl[-m\Psi_{\Gcal
,n}(j\epsilon_{m,n}
)/8\bigr]
\\
&&\quad  \leq \exp\bigl(2m\epsilon_{m,n}^2\bigr) \sum
_{j\geq M_m} \exp \bigl[-cm(j\epsilon_{m,n}
)^{2(p+\alpha)}/8\bigr]
\\
&&\quad  \lesssim \exp\bigl(2m\epsilon_{m,n}^2\bigr) \exp
\bigl[-cm(M_m \epsilon _{m,n})^{2(p+\alpha)}/8\bigr],
\end{eqnarray*}
where the right side of the above display vanishes if
$(M_m\epsilon_{m,n})^{p+\alpha}$ is a sufficiently large multiple of
$\epsilon_{m,n}$.
This holds if we choose $M_m = M\epsilon_{m,n}^{-\vafrac{p+\alpha
-1}{p+\alpha
}}$ for
a large constant $M$. Equation \eqref{Eqn-support-2} also holds.

It remains to verify Equation \eqref{Eqn-support-3}. By Theorem~\ref{Thm-KL},
as long as $n \gtrsim\log(1/\epsilon_{m,n})$,
\begin{eqnarray*}
\log\Pi\bigl(G\in B_K(G_0,\epsilon_{m,n})
\bigr) &\geq& c(c_0) \log\bigl(\epsilon _{m,n}^2/n^3
\bigr)^{kd}
\\
&=& c(c_0) kd(2\log\epsilon_{m,n}- 3\log
n).
\end{eqnarray*}
Equation \eqref{Eqn-support-3} holds for a sufficiently large constant $C$
because $\epsilon_{m,n}> (\log n/m)^{1/2} + (\log m/\allowbreak  m)^{1/2}$, and the
constraint
that $n > \log m$.
\comment{
Moreover, $\Pi(B_{\Hcal}(G_0,2 j\epsilon_{m,n}) \setminus B_{\Hcal
}(G_0,j\epsilon_{m,n}))
\leq\Pi(B_{\Hcal}(G_0,2 j\epsilon_{m,n}))$.
Take any $j \geq M_m$,
if $d_{\mathcal{H}}(G,G_0) \leq j\epsilon_{m,n}$, then at least one
of $G$'s extreme
points is within
$O(j\epsilon_{m,n})$ distance from $G_0$'s extreme points, by Lemma~\ref{Lem-M} (b). By an union bound, and the assumption that
the prior densities for $\vectheta_1,\ldots,\vectheta_k$ are bounded
away from 0,
$\Pi(B_{\Hcal}(G_0,2j\epsilon_{m,n})) \lesssim k^2 (2j\epsilon
_{m,n})^d$. As the
result, the logarithm of the left side of Equation \eqref{Eqn-support-2} is
upper bounded by
\begin{eqnarray*}
\log\bigl[k^2(2j\epsilon_{m,n})^d
\bigl(n^3/\epsilon_{m,n}\bigr)^{kd}\bigr] \leq \log
\bigl(k^2 2^d\bigr) + d\log j + kd \log(1/
\epsilon_{m,n}) + 3kd\log n
\\
\lesssim m (j\epsilon_{m,n})^{2(p+\alpha)} \lesssim m
\Psi_{\Gcal
_m,n}(j\epsilon_{m,n})/16
\end{eqnarray*}
The last inequality of the previous display is due to Theorem~\ref
{Thm-C} (a).
The next to the last inequality holds because for any $j\geq M_m$,
$m(j\epsilon_{m,n})^{2(p+\alpha)} \gtrsim m\epsilon_{m,n}^2
\gtrsim\log n \vee\log(1/\epsilon_{m,n})$, and that
$m(j\epsilon_{m,n})^{2(p+\alpha)} \gtrsim\log j$.
}

Now, we can apply Theorem~\ref{Thm-Gen} to obtain a posterior contraction
rate $M_m\epsilon_{m,n}\asymp\epsilon_{m,n}^{1/(p+\alpha)}$.
\end{pf*}
%
%pa7.subsection.subsubsection.2 #&#
\begin{pf*}{Proof of Theorem \protect\ref{Thm-Main-2}}
The proof proceeds in exactly the same way as Theorem~\ref{Thm-Main},
except that part (b) of Theorem~\ref{Thm-C} is applied instead of part (a).
Accordingly, $p$ is replaced by $1$ in the rate exponent.
\end{pf*}
%pa7.subsection.subsubsection.3 #&#
\begin{pf*}{Proof of Theorem \protect\ref{Thm-minimax} (Minimax lower bounds)}
(a) The proof involves the
construction of a pair of polytopes in $\Gcal^k$ whose set difference has
small volume for a given Hausdorff distance. We consider
two separate cases: (i) $k/2 \leq d$ and (ii) $k > 2d$.

If $k/2 \leq d$, consider a $q = \lfloor k/2 \rfloor$-simplex $G_0$
that is spanned by $q+1$ vertices in general positions.
Take a vertex of $G_0$, say $\vectheta_0$. Construct
$G_0'$ by chopping $G_0$ off by an $\epsilon$-cap
that is obtained by the convex hull of $\vectheta_0$ and $q$ other
points which
lie on the edges adjacent to $\vectheta_0$, and of distance $\epsilon$
from $\vectheta_0$. Clearly, $G_0'$ has $2q \leq k$ vertices, so
both $G_0$ and $G_0'$ are in $\Gcal^k$. We have
$d_{\mathcal{H}}(G_0, G_0') \asymp\epsilon$, and
$\vol_{q}(G_0 \setminus G_0') \asymp\epsilon^{q}$.
Due to assumption (S4),
$V(p_{\veceta|G_0}, p_{\veceta|G_0'}) \lesssim\epsilon^{q+\alpha'}$.
We note
here and for the rest of the proof,
the multiplying constants in asymptotic inequalities depend only
on $r,R,\delta$ of Properties \ref{propA1} and \ref{propA2}.

If $k > 2d$, consider a $d$-dimensional polytope $G_0$ which has $k-d+1$
vertices in general positions.
Construct $G_0'$ in the same way as above (by chopping $G_0$ off by
an $\epsilon$-cap that contains a vertex $\vectheta_0$ which
has $d$ adjacent vertices). Then, $G_0'$ has $(k-d+1)-1+d = k$ vertices.
Thus, both $G_0'$ and $G_0$ are in $\Gcal^k$. We have
$d_{\mathcal{H}}(G_0, G_0') \asymp\epsilon$, and
$\vol_{d}(G_0 \setminus G_0') \asymp\epsilon^{d}$.
Due to assumption (S4),
$V(p_{\veceta|G_0}, p_{\veceta|G_0'}) \lesssim\epsilon^{d+\alpha'}$.

To combine the two cases, let $q = \min(\lfloor k/2 \rfloor,d)$.
We have constructed a pair of $G_0, G_0' \in\Gcal^k$ such that
$d_{\mathcal{H}}(G_0, G_0') \asymp\epsilon$, and
$V(p_{\veceta|G_0}, p_{\veceta|G_0'}) \lesssim\epsilon^{q+\alpha'}$.
By Lemma~\ref{Lem-KL-bound}, $K({p_{\SS_{[n]}|G_0}},
{p_{\SS_{[n]}|G_0'}}) \lesssim
n W_1(p_{\veceta|G_0}, p_{\veceta|G_0'}) \lesssim n V(p_{\veceta
|G_0}, p_{\veceta|G_0'})
\leq C n\epsilon^{q+\alpha'}$ for some constant $C > 0$ independent\vspace*{1pt}
of $\epsilon$ and $n$. Note that the second inequality
in the above display is due to Theorem~6.15 of \cite{Villani-08}.

Applying the method due to Le Cam (cf. \cite{Yu-97}, Lemma~1), for any
estimator
$\hatG\in\Gcal^*$,\vspace*{-1pt}
\[
\max_{G \in\{G_0,G_0'\}} P_{\SS_{[n]}|G_0}d_{\mathcal{H}}(G, \hatG) \gtrsim
\epsilon\biggl(1 - \frac{1}{2}V\bigl(P_{\SS_{[n]}|G_0}^{m},
P_{\SS_{[n]}|G_0'}^{m}\bigr)\biggr).
\]
Here, $P_{\SS_{[n]}|G_0}^{m}$ denotes the (product) distribution of
the $m$-sample
$\SS_{[n]}^1,\ldots,\SS_{[n]}^m$. Thus,\vspace*{-1pt}
\begin{eqnarray*}
V^2\bigl(P_{\SS_{[n]}|G_0}^{m},P_{\SS_{[n]}|G_0'}^{m}
\bigr) & \leq& h^2\bigl(P_{\SS
_{[n]}|G_0}^{m},P_{\SS_{[n]}|G_0'}^{m}
\bigr)
\\[-1pt]
& = & 1 - \int\bigl[P_{\SS_{[n]}|G_0}^{m}P_{\SS_{[n]}|G_0'}^{m}
\bigr]^{1/2}
\\[-1pt]
& = & 1 - \bigl[1- h^2({p_{\SS_{[n]}|G_0}},
{p_{\SS_{[n]}
|G_0'}})\bigr]^{m}
\\[-1pt]
& \leq& 1- \bigl(1- Cn\epsilon^{q+\alpha'}\bigr)^{m}.
\end{eqnarray*}
The last inequality is
due to $h^2({p_{\SS_{[n]}|G_0}},{p_{\SS_{[n]}|G_0'}})
\leq K({p_{\SS_{[n]}|G_0}},{p_{\SS
_{[n]}|G_0'}}) \leq
Cn\epsilon^{q+\alpha'}$.
Thus,\vspace*{-1pt}
\[
\max_{G \in\{G_0,G_0'\}} P_{\SS_{[n]}|G_0}d_{\mathcal{H}}(G, \hatG) \gtrsim
\epsilon\biggl(1 - \frac{1}{2}\bigl[1- \bigl(1- Cn\epsilon^{q+\alpha'}
\bigr)^{m}\bigr]^{1/2}\biggr).
\]

Letting $\epsilon^{q+\alpha'} = \frac{1}{Cm n}$, the right side
of the
previous display is bounded from below by $\epsilon(1-\frac
{1}{2}(1-1/2)^{1/2})$.

(b) We employ the same construction of $G_0$ and $G_0'$ as in part (a).
Using the argument used in the proof of Lemma~\ref{Lem-KL-ez},
$K({p_{\SS_{[n]}|G_0'}}, {p_{\SS_{[n]}|G_0}})
= \int
\log[\vol_{q} G_0/ \vol_{q} G_0']\, \mathrm{d} P_{\veceta|G_0}
\leq\int\log(1+C\epsilon^{q})P_{\veceta|G_0}
\lesssim\epsilon^{q}$. So, $h^2({p_{\SS_{[n]}|G_0}},
{p_{\SS_{[n]}|G_0'}}) \leq K({p_{\SS_{[n]}|G_0'}},
{p_{\SS_{[n]}|G_0}})
\lesssim\epsilon^{q}$. Then, the proof proceeds as in part (a).

(c) Let $G_0'$ be a polytope such that $|\extr G_0'| = |\extr G_0 | =
k$ and
$d_{\mathcal{H}}(G_0',G_0) =\epsilon$.
By Lemma~\ref{Lem-diff-1}, $\vol_{p}(G_0\bigtriangleup G_0') =
\mathrm{O}(\epsilon)$,
where $p=(k-1)\wedge d$.
The proof proceeds as in part (a) to obtain
$(1/mn)^{1/(1+\alpha')}$ rate for the lower bound
under assumption (S4). Under assumption (S4$'$), as in part (b),
the dependence on $n$ can be removed to obtain $1/m$ rate.
\end{pf*}
%pa7.subsection.subsubsection.4 #&#
\begin{pf*}{Proof of $\alpha$-regularity of the Dirichlet-induced
densities in Lemma \protect\ref{Lem-Dirichlet}}
First, consider the case $k \leq d+1$.
For $\veceta^* \in G$, write $\veceta^* = \beta_1^* \vectheta_1
+\cdots+ \beta_k^* \vectheta_k$.
For $\vecbeta\in\Delta^{k-1}$ such that $|\beta_i-\beta_i^*| \leq
\epsilon/k$ for
all $i=1,\ldots, k-1$, we have
$\|\veceta- \veceta^*\| = \|\sum_{i=1}^{k}(\beta_i - \beta
_i^*)\vectheta_i\|
\leq\sum_{i=1}^{k} |\beta_i-\beta_i^*| \leq2\sum_{i=1}^{k-1}
|\beta_i-\beta_i^*|
\leq2\epsilon$. Here, we used the fact that
$\|\vectheta_i\| \leq1$ for any $\vectheta_i \in\Delta^{d}$.
Without loss of generality,
assume that $\beta_k^* \geq1/k$. Then, for any $\epsilon< 1/k$\vspace*{-1pt}
\begin{eqnarray*}
&& P_{\veceta|G}\bigl(\bigl \|\veceta-\veceta^*\bigr \|\leq2\epsilon\bigr) \\
&&\quad \geq
P_{\vecbeta}\bigl(\bigl |\beta_i - \beta_i^*\bigr | \leq
\epsilon/k; i=1,\ldots , k-1\bigr)
\\
&&\quad  = \frac{\Gamma(\sum\gamma_i)}{  \prod_{i} \Gamma(\gamma_i)} \int_{\beta_i \in[0,1];|\beta_i-\beta_i^*| \leq\epsilon/k;
i=1,\ldots,k-1} \prod
_{i=1}^{k-1}\beta_i^{\gamma_i-1}
\Biggl(1-\sum_{i=1}^{k-1}\beta_i
\Biggr)^{\gamma_k-1}\,\mathrm{d}\beta_1\cdots \,\mathrm{d}\beta _{k-1}
\\
&&\quad  \geq\frac{\Gamma(\sum\gamma_i)}{ \prod_{i} \Gamma(\gamma_i)} \prod_{i=1}^{k-1}
\int_{\max(\gamma_i^*-\epsilon/k,0)}^{\min
(\gamma_i^*+\epsilon/k,1)} \beta_i^{\gamma_i-1} \,\mathrm{d}
\beta_i \geq\frac{\Gamma(\sum\gamma_i)}{ \prod_{i} \Gamma(\gamma_i)} (\epsilon/k)^{k-1}.
\end{eqnarray*}
Both the second and the third inequality in the previous display
exploits the
fact that since $\gamma_i \leq1$, $x^{\gamma_i-1} \geq1$ for any $x
\leq1$.

Now, consider the case $k > d+1$. The proof in the previous case applies,
but we can achieve a better lower bound
because the intrinsic dimensionality of $G$ is $d$, not $k-1$.
Since $\veceta^* \in\conv(\vectheta_1,\ldots,\vectheta_k) \subset
\Delta^{d}$,
by Carath\'eodory's theorem, $\veceta^*$ is the convex combination of
$d+1$ or fewer
extreme points among $\vectheta_i$'s. Without loss of generality, let
$\vectheta_1,\ldots,
\vectheta_{d+1}$ be such points, and write $\veceta^* = \beta
_1^*\vectheta_1 +
\cdots+\beta_{d+1}^* \vectheta_{d+1}$.
Consider $\veceta= \beta_1 \vectheta_1 + \cdots+ \beta_k\vectheta
_k$, where
$\|\beta_i-\beta_i^*| \leq\epsilon/k$, for $i=1,\ldots, d$,
while $0\leq\beta_i \leq\epsilon/k$ for $i = d+2,\ldots, k$. Then,
$\|\veceta- \veceta^*\| \leq2\epsilon$. This implies that
\begin{eqnarray*}
&&P_{\veceta|G}\bigl(\bigl \|\veceta-\veceta^*\bigr \|\leq2\epsilon\bigr) \\
&&\quad \geq
P_{\vecbeta}\bigl(\bigl |\beta_i - \beta_i^*\bigr | \leq
\epsilon/k, i=1,\ldots , d+1; |\beta_j| \leq\epsilon/k, j > d+1
\bigr)
\\
&&\quad \geq\frac{\Gamma(\sum\gamma_i)}{  \prod_{i} \Gamma(\gamma_i)} \prod_{i=1}^{d}
\int_{\max(\gamma_i^*-\epsilon/k,0)}^{\min(\gamma
_i^*+\epsilon/k,1)} \beta_i^{\gamma_i-1} \,\mathrm{d}
\beta_i \prod_{i=d+2}^{k}\int
_{0}^{\epsilon/k} \beta_i^{\gamma_i-1} \,\mathrm{d}
\beta_i
\\
&&\quad \geq\frac{\Gamma(\sum\gamma_i)}{  \prod_{i} \Gamma(\gamma_i)} (\epsilon/k)^{d + \sum_{i=d+2}^{k}\gamma_i}\biggl/\prod
_{i=d+2}^{n}\gamma_i \gtrsim
\epsilon^{d + \sum_{i=1}^{k}\gamma_i}.
\end{eqnarray*}
This concludes the proof.
\end{pf*}

\comment{
The linear map $\vecbeta\mapsto\veceta$ from $\Delta^{k-1}$ to
$\Delta^{d}$ is generally not invertible.
Write $\veceta= \beta_1(\vectheta_1-\vectheta_k) + \cdots+ \beta
_{k-1}(\vectheta_{k-1}-\vectheta_k) + \vectheta_k
= L' \vecbeta+ \vectheta_k$, where $L' = [\vectheta_1-\vectheta_k
\ldots\vectheta_{k-1}-\vectheta_k]$
is a $(d+1)\times(k-1)$ matrix. Since $\vectheta_i \in\Delta^d$,
let $L$ be
the $d\times(k-1)$ obtained by deleting the bottom row from $L'$.
}
\begin{appendix}
%s8 #&#
\section{Proofs of geometric lemmas}\label{appA}

%pa8.subsection.subsubsection.1 #&#
\begin{pf*}{Proof of Lemma \protect\ref{Lem-M}}
(a) Let $G= \conv(\vectheta_1,\ldots,\vectheta_k)$
and $G' = \conv(\vectheta'_1,\ldots, \vectheta'_{k'})$.
This part of the lemma is immediate from the definition by noting that
for any $x\in G$, $d(x,G') \leq\min_{j} \|x-\vectheta'_j\|$, while
the maximum of $d(x,G')$ is attained at some extreme point of $G$.

(b) Let $d_{\mathcal{H}}(G,G') = \epsilon$ for some small $\epsilon> 0$.
Take an extreme point of $G$, say $\vectheta_1$.
Due to \ref{propA2}, there is a ray emanating
from $\vectheta_1$ that intersects with the interior of $G$ and the angles
formed by the ray and all (exposed) edges incident to
$\vectheta_1$ are bounded from above by $\pi/2-\delta$.
Let $x$ be the intersection between the ray and the boundary of
$B_{\d}(\vectheta_1,\epsilon)$.

Let $H$ be a $\d-1$-dimensional
hyperplane in $\real^{\d}$ that touches (intersects with) $B_{\d
}(\vectheta_1,\epsilon)$
at only $x$. Define $C(x)$, resp. $C_{\epsilon}(x)$, to
be the $\d$-dimensional caps obtained by the intersection between
$G$, resp. $G_{\epsilon}$, with the half-space which contains
$\vectheta_1$
and which is supported by $H$. For any $x'$ that lies in the intersection
of $H$ and a line segment $[\vectheta_1,\vectheta_i]$, where
$\vectheta_i$ is another vertex of $G$,
the line segment $[x,x'] \in H$ and $\|x-x'\| \leq\epsilon\cot\delta$.
Suppose that the ray emanating from $x$ through $x'$
intersects with $\bd G_{\epsilon}$ at $x''$.
Then, $\|x'-x''\| \leq\epsilon/\sin\delta$, which implies
that $\|x-x''\| \leq\epsilon(\cot\delta+ 1/\sin\delta)$ by
triangle inequality.
This entails that $\operatorname{Diam}C_{\epsilon}(x) \leq C \epsilon
$, where
$C = (1+(\cot\delta+ 1/\sin\delta)^2)^{1/2}$.

Now, $d_{\mathcal{H}}(G,G') = \epsilon$ implies that
$G' \cap B_{\d}(\vectheta_1,\epsilon) \neq\emptyset$. There is
an extreme point of $G'$ in the half-space which contains $B(\vectheta
_1,\epsilon)$ and
is supported by $H$.
But $G' \subset G_{\epsilon}$, so there is an extreme point of $G'$ in
$C_{\epsilon}(x)$. Hence, there is $\vectheta'_j \in G'$ such that
$\|\vectheta'_j - \vectheta_1\| \leq\operatorname{Diam}(C_{\epsilon
}(x)) \leq C\epsilon$.
Repeat this argument for all other extreme points of $G$ to conclude
that $d_{\mathcal{M}}(G,G') \leq C\epsilon$.
\end{pf*}
%
%pa8.subsection.subsubsection.2 #&#
\begin{pf*}{Proof of Lemma \protect\ref{Lem-diff-1}}
(a) Let $d_{\mathcal{H}}(G,G') = \epsilon$. There exists either a
point $x \in
\bd G$ such
that $G' \cap B_{\d}(x,\epsilon/2) = \emptyset$, or a point $x' \in
\bd G'$
such that $G \cap B_{\d}(x,\epsilon/2) = \emptyset$. Without loss of
generality,
assume the former. Thus, $\vol_{\d} G \bigtriangleup G' \geq\vol
_{\d} B_{\d}(x,\epsilon/2)
\cap G$. Consider the convex cone emanating from $x$ that circumscribes
the $\d$-dimensional spherical ball $B_\d(\vectheta_c,r)$ (whose
existence is
given by Property \ref{propA1}). Since $\|x-\vectheta_c\| \leq R$,
the angle between the line segment $[x,\vectheta_c]$ and
the cone's rays is bounded from below by $\sin\varphi\geq r/R$.
So, $\vol_{\d} B_d(x,\epsilon/2) \cap G \geq c_1 \epsilon^{\d}$,
where $c_1$ depends only on $r,R, p$.

(b) Let $d_{\mathcal{H}}(G,G') = \epsilon$. Then $G' \subset
G_{\epsilon}$ and $
G \subset G_{\epsilon}'$. Take any point $x \in\bd G$, let $x'$ be
the intersection between $\bd G_{\epsilon}$ and the ray emanating from
$\vectheta_c$ and passing through $x$. Let $H_1$ be a $\d-1$
dimensional supporting hyperplane for $G$ at $x$. There is also a supporting
hyperplane $H_2$ of $G'$ that is parallel to $H_1$ and of at most
$\epsilon$ distance away from $H_1$.
Since $\|\vectheta_c - x\| \leq R$, while the distance from $\vectheta
_c$ to
$H_1$ is lower bounded by $r$, the angle $\varphi$ between vector
$\vectheta_c-x$ and the vector normal to $H_1$ satisfies $\cos\varphi
\geq r/R$.
This implies that $\|x'-x\| \leq\epsilon/\cos\varphi\leq\epsilon R/r$,
so $\|x'-\vectheta_c\|/\|x-\vectheta_c\| \leq1+\epsilon R/r^2$. In
other words,
$G_{\epsilon} - \vectheta_c \subset(1+\epsilon R/r^2)(G-\vectheta_c)$.
So, $\vol_{\d} G' \setminus G \leq
\vol_{\d} G_{\epsilon} \setminus G \leq[(1+\epsilon R/r^2)^{\d}-1]
\vol_{\d} G
\leq C_1 \epsilon$, where $C_1$ depends only on $r,R,p$.
We obtain a similar bound for $\vol_{\d} G \setminus G'$, which concludes
the proof.
\end{pf*}
%
%pa8.subsection.subsubsection.3 #&#
\begin{pf*}{Proof of Lemma \protect\ref{Lem-diff-3}}
We provide a proof for case (a). Let $G = \conv(\vectheta_1, \ldots,
\vectheta_k)$ and
$G' = \conv(\vectheta'_1, \ldots, \vectheta'_k)$, where $G$ is fixed
but $G'$ is allowed to vary. Since $G$ is fixed, it satisfies
\ref{propA1} and \ref{propA2} for some constants $r, R$ and $\delta$ (depending on $G$).
Moreover, there is some $\epsilon_0 = \epsilon_0(G)$ depending
only $G$ such that as soon as $d_{\mathcal{H}}(G,G') \leq\epsilon_0$,
$G'$ also satisfies \ref{propA1} and \ref{propA2} for constants $\delta'=\delta/2,
r'=r/2, R'=2R$.

Suppose that $d_{\mathcal{H}}(G,G') = \epsilon$ such that $\epsilon<
\epsilon_0$.
By Lemma~\ref{Lem-M}(b) for each vertex of $G$, say $\vectheta_i$,
there is a vertice of $G'$, say $\vectheta'_i$, such that
$\vectheta'_i \in B_{\d}(\vectheta_i,C_0\epsilon)$ with $C_0 = C_0(G)$
depending only on~$\delta$.
Moreover, there is at least one vertice of $G$, say
$\vectheta_1$, for which $\|\vectheta'_1 - \vectheta_1\| \geq
\epsilon$.

There are only three possible general positions for $\vectheta'_1$ relatively
to $G$. Either
\begin{enumerate}[(iii)]
\item[(i)] $\vectheta'_1 \in G$, or
\item[(ii)] $\vectheta'_1 \in2\vectheta_1 - G$, or
\item[(iii)] $\vectheta'_1$ lies in a cone formed by all half-spaces
supported by the $\d-1$ dimensional faces adjacent to $\vectheta_1$.
Among these
there is one half-space that contains $G$, and one that does not
contain $G$.
\end{enumerate}
If (i) is true, by Property \ref{propA1}, $G$ has at least one face $S \supset
\vectheta_1$
such that the distance from $\vectheta'_1$ to the hyperplane that
provides support
for $S$ is bounded from below by $\epsilon r/R$.
Let $B \subset S$ be a
homothetic transformation of $S$ with respect to center $\vectheta_1$
that maps $x \in S$ to $\tilde{x} \in B$ such that the ratio
$\eta:= \|\vectheta_1 - \tilde{x}\|/
\|\vectheta_1 - x\|$ satisfies
$1-\eta= 2C_0\epsilon/\min_{i\neq j}\|\vectheta_i-\vectheta_j\|
\in(0,1/2)$.
This is possible
as soon as $\epsilon< \min_{i\neq j} \|\vectheta_i -\vectheta_j\|/4C_0$.
Then, for any $\vectheta_j \in S$, $j\neq1$,
under this transformation $\vectheta_j \mapsto\tilde{\vectheta}_j
\in S$
for which $\|\tilde{\vectheta}_j - \vectheta_j\| = (1-\eta)\|
\vectheta_1-\vectheta_j\|
\geq2C_0\epsilon$. Since
$\|\vectheta'_j - \vectheta_j\| \leq C_0\epsilon$, the construction
of $B$ implies that $\vectheta'_j \notin B$. As a result,
$B \cap G' = \emptyset$. Moreover,
$\vol_{\d-1} B = \eta^{p-1} \vol_{\d-1} S \geq(1/2)^{p-1}\vol
_{\d-1} S \geq c_0(G)$, a
constant depending only on $G$.
Let $Q$ be a $\d$-pyramid which has
apex $\vectheta'_1$ and base $B$. It follows that
$\interior Q \cap\interior G' = \emptyset$, which implies
that $\interior Q \subset G\setminus G'$ ($\interior$ stands for the
relative interior of a set).
Hence, $\vol_{\d}G\setminus G' \geq
\vol_{\d}Q \geq\frac{1}{\d}\epsilon r/R \vol_{\d-1}B
\geq\frac{1}{\d}\epsilon c_0(G) r/R$.

If (ii) is true, the same argument can be applied to show that
$\vol_{\d}(G' \setminus G) = \Omega(\epsilon)$.
If (iii) is true,
a similar argument continues to apply: we obtain a lower bound for either
$\vol_{\d} G' \setminus G$ or $\vol_{\d} G\setminus G'$.
$G$ has a face (supported by a hyperplane, say, $H$) such that
the distance from $\vectheta'_1$ to $H$ is $\Omega(\epsilon)$.
If the
half-space supported by $H$ that contains $\vectheta'_1$ but does not
contain $G$,
then $\vol_{\d} G'\setminus G = \Omega(\epsilon)$.
If, on the other hand, the associated half-space does contain $G$, then
$\vol_{\d} G\setminus G' = \Omega(\epsilon)$.
The proof for case (b) is similar and is omitted.
\end{pf*}

%s9 #&#
\section{Proof of abstract posterior contraction theorem}
\label{appB}

\setcounter{equation}{15}
\renewcommand{\theequation}{\arabic{equation}}

A key ingredient in the general analysis of convergence of posterior
distributions is
through establishing the existence of tests for subsets of parameters
of interest.
A test $\varphi_{m,n}$ is a measurable indicator function of the
$m\times n$-sample
$\SS_{[n]}^{[m]}= (\SS_{[n]}^{1},\ldots,\SS_{[n]}^{m})$ from an
admixture model. For
a fixed pair of convex
polytopes $G_0,G_1\in\Gcal$, where $\Gcal$ is a given subset
of $\Delta^d$,
consider tests for discriminating $G_0$
against a closed Hausdorff ball centered at $G_1$.
The following two lemmas on the existence of tests highlight the fundamental
role of the Hellinger information:

%le9 #&#
\begin{lemma}
\label{Lem-test-ball}
Fix a pair of $(G_0,G_1) \in(\Gcal^*\times\Gcal)$ and let $\delta=
d_{\mathcal{H}}(G_0,G_1)$.
Then, there exist tests $\{\varphi_{m,n}\}$ that have the following properties:
%
%e17 #&#
%e18 #&#
\begin{eqnarray}
\label{Eqn-power-1} P_{\SS_{[n]}|G_0}^{m} \varphi_{m,n}& \leq& D
\exp\bigl[-m\Psi_{\Gcal
,n}(\delta)/8\bigr],
\\
\label{Eqn-power-2}\sup_{G\in\Gcal\cap B_{d_{\mathcal{H}}}(G_1,\delta/2)} P_{\SS
_{[n]}|G}^{m} (1-
\varphi_{m,n}) & \leq& \exp\bigl[-m\Psi_{\Gcal,n} (\delta)/8\bigr].
\end{eqnarray}
Here, $D :=
D (\Phi_{\Gcal,n}(\delta), \Gcal\cap
B_{d_{\mathcal{H}}}(G_1,\delta/2), d_{\mathcal{H}} )$, i.e.,
the maximal number of elements in $\Gcal\cap B_{d_{\mathcal
{H}}}(G_1,\delta/2)$
that are
mutually separated by at least $\Phi_{\Gcal,n}(\delta)$ in
Hausdorff metric $d_{\mathcal{H}}$.
\end{lemma}

\begin{pf}
We begin the proof by noting that a direct application of standard
results on existence of tests (cf. \cite{LeCam-86}, Chapter~4) is not
possible, due to the lack of convexity of the space of densities
of $\SS_{[n]}$ as $G$ varies in some subset $\Gcal\subset\Gcal^*$,
even if $\Gcal$ is convex. This difficulty is overcome by appealing to
a packing argument.

Consider a maximal $\Phi_{\Gcal,n}(\delta)$-packing in $d_{\mathcal{H}}$
metric for the set $\Gcal\cap B_{d_{\mathcal{H}}}(G_1,\delta/2)$.
This yields a set of $D = D(\Phi_{\Gcal,n}(\delta),
\Gcal\cap B_{d_{\mathcal{H}}}(G_1,\delta/2),
d_{\mathcal{H}})$ elements $\tilde{G}_1,\ldots,\tilde{G}_D \in
\Gcal\cap
B_{d_{\mathcal{H}}}(G_1,\delta/2)$.

Next, we note the following fact: for any $t=1,\ldots, D$, if $G \in
\Gcal\cap
B_{d_{\mathcal{H}}}(G_1,\delta/2)$ and $d_{\mathcal{H}}(G,\tilde
{G}_t) \leq\Phi_{\Gcal
,n}(\delta)$,
then by the definition of $\Phi$,
$h^2({p_{\SS_{[n]}|G}},p_{S_{[n]}|\tilde{G}_t}) \leq
\frac{1}{4}\Psi_{\Gcal
,n}(\delta)$.
By the definition of Hellinger information,
$h^2({p_{\SS_{[n]}|G_0}}, p_{S_{[n]}|\tilde{G}_t}) \geq
\Psi_{\Gcal,n}(\delta
)$. Thus,
by triangle inequality, $h({p_{\SS_{[n]}|G_0}},
{p_{\SS_{[n]}|G}}) \geq
%h(p_{G_0}, p_{\tilde{G}_t}) - h(p_{G}, p_{\tilde{G}_t}) \geq
\frac{1}{2}\Psi_{\Gcal,n}(\delta)^{1/2}$.

For each pair of $G_0,\tilde{G}_t$ there exist tests $\omega_n^{(t)}$
of ${p_{\SS_{[n]}|G_0}}$ versus the Hellinger ball $\Pcal
_2(t) :=
\{{p_{\SS_{[n]}|G}}| G \in\Gcal^*; h({p_{\SS_{[n]}
|G}}, p_{S_{[n]}|\tilde{G}_t}) \leq
\frac{1}{2}h({p_{\SS_{[n]}|G_0}}, p_{S_{[n]}|\tilde
{G}_t})\}$ such that,
\begin{eqnarray*}
P_{\SS_{[n]}|G_0}^{m} \omega_{m,n}^{(t)} &\leq&\exp
\bigl[-m h^2({p_{\SS_{[n]}
|G_0}}, p_{S_{[n]}|\tilde{G}_t})/8
\bigr],
\\
\sup_{P_2 \in\Pcal_2(t)} P_2^m \bigl(1-
\omega_{m,n}^{(t)}\bigr) &\leq&\exp\bigl[-m h^2(
{p_{\SS_{[n]}|G_0}}, p_{S_{[n]}|\tilde{G}_t})/8\bigr].
\end{eqnarray*}

Consider the test $\varphi_{m,n} = \max_{1\leq t\leq D} \omega
_{m,n}^{(t)}$, then
\begin{eqnarray*}
P_{\SS_{[n]}|G_0}^{m} \varphi_{m,n} & \leq& D \times\exp
\bigl[-m \Psi _{\Gcal
,n}(\delta)/8\bigr],
\\
\sup_{G \in\Gcal\cap B_{d_{\mathcal{H}}}(G_1,\delta/2)} P_{\SS
_{[n]}|G}^{m} (1-
\varphi_{m,n}) & \leq& \exp\bigl[-m \Psi_{\Gcal,n}(\delta)/8\bigr].
\end{eqnarray*}
The first inequality is due to $\varphi_{m,n} \leq\sum_{t=1}^{D}\omega_{m,n}^{(t)}$,
and the second is due to the fact that for any $G \in\Gcal\cap
B_{d_{\mathcal{H}}
}(G_1,\delta/2)$
there is some $d = 1,\ldots, D$ such that $d_{\mathcal{H}}(G,\tilde
{G}_t) \leq
\Phi_{\Gcal,n}(\delta)$,
so that ${p_{\SS_{[n]}|G}}\in\Pcal_2(t)$.
\end{pf}

Next, the existence of tests can be shown for discriminating $G_0$ against
the complement of a closed Hausdorff ball:

%le10 #&#
\begin{lemma}
\label{Lem-ballcomplement}
Let $G_0 \in\Gcal^*$ and $\Gcal\subset\Gcal^*$.
Suppose that for some non-increasing function $D(\epsilon)$, some
$\epsilon_{m,n}\geq0$ and every $\epsilon> \epsilon_{m,n}$,
%
%e19 #&#
\begin{eqnarray}\label{Eqn-entropy-1}
&&\sup_{G_1\in\Gcal} D\bigl(\Phi_{\Gcal,n}(\epsilon), \Gcal\cap
B_{d_{\mathcal{H}}}(G_1,\epsilon/2), d_{\mathcal{H}}\bigr)
\nonumber
\\[-8pt]\\[-8pt]
&&\hphantom{\sup_{G_1\in\Gcal}}{}\times D\bigl(\epsilon/2, \Gcal\cap B_{d_{\mathcal{H}}}(G_0,2
\epsilon)\setminus B_W(G_0,\epsilon), %\{G \in\Gcal: \epsilon\leq d_{\rho}(G_0,G) \leq2\epsilon\},
d_{\mathcal{H}}\bigr) \leq D(\epsilon). \nonumber
\end{eqnarray}
Then, for every $\epsilon> \epsilon_{m,n}$, and any $t_0 \in\Nat$,
there exist tests $\varphi_{m,n}$ (depending on $\epsilon> 0$) such that
%
%e20 #&#
%e21 #&#
\begin{eqnarray}
\label{Eqn-type-1b} P_{G_0} \varphi_{m,n} & \leq& D(\epsilon) \sum
_{t=t_0}^{\lceil\operatorname{Diam}(\Gcal)/\epsilon\rceil
}\exp\bigl[-m\Psi_{\Gcal,n}(t
\epsilon)/8\bigr],
\\
\label{Eqn-type-2b} \sup_{G\in\Gcal: d_{\mathcal{H}}(G_0,G) > t_0\epsilon} P_{G} (1-\varphi
_{m,n}) & \leq& \exp\bigl[-m \Psi_{\Gcal,n}(t_0
\epsilon)/8\bigr].
\end{eqnarray}
\end{lemma}

\begin{pf}
The proof consists of a standard peeling device (e.g., \cite
{Ghosal-Ghosh-vanderVaart-00})
and a packing argument as in the previous proof.
For a given $t \in\Nat$ choose a maximal $t\epsilon/2$-packing for
set $S_t = \{G\dvt t\epsilon< d_{\mathcal{H}}(G_0,G) \leq(t+1)\epsilon
\}$.
This
yields a set $S'_t$ of at most $D(t\epsilon/2,S_t,d_{\mathcal{H}})$ points.
Moreover, every $G\in S_t$ is within distance $t\epsilon/2$ of at least
one of the points in $S'_t$. For every such point $G_1\in S'_t$, there
exists a test
$\omega_{m,n}$ satisfying Equations \eqref{Eqn-power-1} and \eqref{Eqn-power-2},
where $\delta$ is taken to be $\delta=t\epsilon$.
Take $\varphi_{m,n}$ to be the maximum of all tests attached this way to
some point $G_1\in S'_t$ for some $t\geq t_0$. Note that
$G \in\Gcal\subset\Delta^d$, so $t \leq\lceil\operatorname
{Diam}(\Gcal)/\epsilon\rceil$.
Then, by union bound, and the condition that $D(\epsilon)$ is non-increasing,\vspace*{-1pt}
\begin{eqnarray*}
P_{\SS_{[n]}|G_0}^{m} \varphi_{m,n} & \leq& \sum
_{t=t_0}^{\lceil\sfrac{\operatorname{Diam}(\Gcal
)}{\epsilon} \rceil} \sum_{G_1 \in S'_t} D
\bigl(\Phi_{\Gcal,n}(t\epsilon), \Gcal\cap B_{d_{\mathcal{H}}}(G_1,t
\epsilon/2), d_{\mathcal{H}} \bigr)
\\
&&\hphantom{\sum
_{t=t_0}^{\lceil\sfrac{\operatorname{Diam}(\Gcal
)}{\epsilon} \rceil} \sum_{G_1 \in S'_t}}{} \times\exp\bigl[-m \Psi_{\Gcal,n}(t\epsilon)/8\bigr]
\\
& \leq& D(\epsilon)\sum_{t\geq t_0}\exp\bigl[-m
\Psi_{\Gcal
,n}(t\epsilon)/8\bigr],
\end{eqnarray*}
and\vspace*{-1pt}
\begin{eqnarray*}
\sup_{G\in\bigcup_{u\geq t_0}S_{u}} P_{\SS_{[n]}|G}^{m} (1-
\varphi_n) & \leq& \sup_{u\geq t_0} \exp\bigl[-m
\Psi_{\Gcal,n}(u\epsilon)/8\bigr]
\\
& \leq& \exp\bigl[-m\Psi_{\Gcal,n}(t_0\epsilon)/8\bigr],
\end{eqnarray*}
where the last inequality is due the monotonicity of $\Psi_{\Gcal
,n}(\cdot)$.%\vadjust{\goodbreak}
\end{pf}
%
%pa9.subsection.subsubsection.1 #&#
%
\begin{pf*}{Proof of the abstract posterior contraction theorem
(Theorem \protect\ref{Thm-Gen})}
In this proof, to simplify notations denote $P_{G} := P_{\SS_{[n]}|G}$.
By a result of Ghosal \textit{et al.} \cite{Ghosal-Ghosh-vanderVaart-00} (Lemma~8.1, page 524),
for every $\epsilon> 0, C> 0$ and every probability measure $\Pi_0$
supported on the set
$B_K(G_0,\epsilon)$ defined by Equation \eqref{Eqn-kl-neighborhood}, we have,
\[
P_{G_0} \Biggl(\int\prod_{i=1}^{m}
\frac{p_{G}(\SS
_{[n]}^i)}{p_{G_0}(\SS_{[n]}
^i)} \,\mathrm{d}\Pi_0(G) \leq\exp\bigl(-(1+C)m
\epsilon^2\bigr) \Biggr) \leq\frac{1}{C^2m\epsilon^2}.
\]
This entails that, by fixing $C= 1$, there is an event $A_m$
with $P_{G_0}^m$-probability at least $1-(m\epsilon_{m,n}^2)^{-1}$, for
which there holds:
%
%e22 #&#
\begin{equation}
\label{Eqn-denominator} \int\prod_{i=1}^{n}p_G
\bigl(\SS_{[n]}^i\bigr)/p_{G_0}\bigl(
\SS_{[n]}^i\bigr) \,\mathrm{d}\Pi(G) \geq\exp\bigl(-2m
\epsilon_{m,n}^2\bigr)\Pi\bigl(B_K(G_0,
\epsilon_{m,n})\bigr).
\end{equation}
Let $\Ocal_m = \{G \in\mathcal{G}^*\dvt  d_{\mathcal{H}}(G_0,G) \geq
M_m \epsilon_{m,n}\}$.
%$S_{m,j} = \{G\in\Gcal_m: \dH(G_0,G)\in[j\epsmn, (j+1)\epsmn)\}$
%for each $j\geq1$.
Due to Equation \eqref{Eqn-entropy-2}, the condition specified by Lemma~\ref
{Lem-ballcomplement}
is satisfied by setting $D(\epsilon) = \exp(m\epsilon_{m,n}^2)$
(constant in
$\epsilon$).
Thus, there exist tests $\varphi_{m,n}$ for which Equations \eqref
{Eqn-type-1b}
and \eqref{Eqn-type-2b} hold with respect to $\Gcal= \supp\Pi$ and
the given $G_0$.
Then,
\begin{eqnarray*}
  P_{G_0}\Pi\bigl(G\in\Ocal_m|\SS_{[n]}^{[m]}
\bigr)
   &=&  P_{G_0}\bigl[\varphi_{m,n}\Pi\bigl(G\in
\Ocal_m|\SS_{[n]}^{[m]}\bigr)\bigr] +
P_{G_0} \bigl[(1-\varphi_{m,n})\Pi\bigl(G\in
\Ocal_m|\SS_{[n]}^{[m]}\bigr)\bigr]
\\
  &\leq& P_{G_0}\bigl[\varphi_{m,n}\Pi\bigl(G\in
\Ocal_m|\SS_{[n]}^{[m]}\bigr)\bigr] +
P_{G_0}\indicator\bigl(A_m^c\bigr)
\\
&&{}+ P_{G_0} \bigl[(1-\varphi_{m,n})\Pi\bigl(G\in
\Ocal_m|\SS_{[n]}^{[m]} \bigr)\indicator(A_m)
\bigr].
\end{eqnarray*}

Applying Lemma~\ref{Lem-ballcomplement}, the first term in the preceding
display is bounded above by
\[
P_{G_0}\varphi_{m,n}\leq D(\epsilon_{m,n})\sum
_{j\geq M_m} \exp\bigl[-m\Psi_{\Gcal,n}(j\epsilon
_{m,n})/8\bigr] \rightarrow0,
\]
thanks to Equation \eqref{Eqn-rate-cond}.
The second term in the above display is bounded by $(m\epsilon
_{m,n}^2)^{-1}$ by
the definition of $A_m$, so this term vanishes.
It remains to show that third term in the display also vanishes as
$m\rightarrow\infty$. By Bayes' rule,
\[
\Pi\bigl(G\in\Ocal_m|\SS_{[n]}^{[m]}\bigr) =
\frac{  \int_{\Ocal_m}\prod_{i=1}^{m}p_G(\SS_{[n]}^i)/p_{G_0}(\SS_{[n]}^i)
\,\mathrm{d}\Pi(G)}{  \int\prod_{i=1}^{m}p_G(\SS_{[n]}^i)/p_{G_0}(\SS_{[n]}^i)
\,\mathrm{d}\Pi(G)},
\]
and then obtain a lower bound for the denominator by Equation \eqref
{Eqn-denominator}.
For the nominator, by Fubini's theorem:
%%%%%%%%
\comment{
%
%e23 #&#
\begin{eqnarray}
&& P_{G_0} \int_{\Ocal_m \cap\Gcal} (1-\varphi_{m,n})
\prod_{i=1}^{m} p_G\bigl(
\SS_{[n]}^i\bigr)/p_{G_0}\bigl(
\SS_{[n]}^i\bigr) d\Pi(G)
\nonumber
\\
& = & P_{G_0} \sum_{j\geq M_m} \int
_{S_{m,j}}(1-\varphi_{m,n}) \prod
_{i=1}^{m} p_G\bigl(
\SS_{[n]}^i\bigr)/p_{G_0}\bigl(
\SS_{[n]}^i\bigr) d\Pi(G)
\nonumber
\\
& = & \sum_{j\geq M_m} \int_{S_{m,j}}
P_{G}(1-\varphi_{m,n}) d\Pi(G)
\nonumber
\\
\label{Eqn-T1} & \leq& \sum_{j\geq M_m}
\Pi(S_{m,j})\exp\bigl[-m\Psi_{\Gcal
,n}(j\epsilon_{m,n})/8
\bigr],
\end{eqnarray}}
%
%e24 #&#
\begin{eqnarray}\label{Eqn-T1}
&& P_{G_0} \int_{\Ocal_m \cap\Gcal} (1-\varphi_{m,n})
\prod_{i=1}^{m} p_G\bigl(
\SS_{[n]}^i\bigr)/p_{G_0}\bigl(
\SS_{[n]}^i\bigr) \,\mathrm{d}\Pi(G)
\nonumber
\\[-8pt]\\[-8pt]
 &&\quad  =  \int_{\Ocal_m \cap\Gcal} P_{G}(1-
\varphi_{m,n}) \,\mathrm{d}\Pi(G) \leq\exp\bigl[-m\Psi_{\Gcal,n}(M_m
\epsilon_{m,n})/8\bigr],\nonumber
\end{eqnarray}
where the last inequality is due to Equation \eqref{Eqn-type-2b}.%\vadjust{\goodbreak}
\comment{
In addition, by \eqref{Eqn-support-1},
%
%e25 #&#
\begin{eqnarray}
&& P_{G_0} \int_{\Ocal_m \setminus\Gcal_m}(1-\varphi_{m,n})
\prod_{i=1}^{m} p_G\bigl(
\SS_{[n]}^i\bigr)/p_{G_0}\bigl(
\SS_{[n]}^i\bigr) d\Pi(G)
\nonumber
\\
\label{Eqn-T2} & = & \int_{\Ocal_m \setminus\Gcal_m} P_G (1-
\varphi_{m,n}) d\Pi(G) \leq\Pi\bigl(\Gcal^* \setminus\Gcal_m
\bigr). %= o(\exp(-2m\epsmn^2)\Pi(B_K(G_0,\epsmn))).
\end{eqnarray}
}

Now, combining bounds \eqref{Eqn-T1} and \eqref{Eqn-denominator} with
condition
\eqref{Eqn-support-2}, we obtain:
\begin{eqnarray*}
&& P_{G_0}(1-\varphi_{m,n})\Pi\bigl(G\in\Ocal_m|
\SS _{[n]}^{[m]}\bigr)\indicator(A_m)
\\
&&\quad \leq \frac{
\exp[-m\Psi_{G_0,n}(\Gcal_m,M_m\epsilon_{m,n})/8]}{\exp
(-2m\epsilon_{m,n}^2)\Pi
(B_K(G_0,\epsilon_{m,n}))}. %& \leq&
%o(1) +
\end{eqnarray*}
The upper bound in the preceeding
display converges to 0 by Equation \eqref{Eqn-support-2}, thereby
concluding the proof.
\end{pf*}

\end{appendix}

\section*{Acknowledgements}
This research was supported in part by NSF Grants CCF-1115769 and OCI-1047871.
The author wish to thank the anonymous referees for helpful comments, and
Qiaozhu Mei and Jian Tang for stimulating discussions
which helped to motivate this work.

%suskaldyti doi

% imsref loaded by jurgita.kaciuliene, 2014-02-07 10:48:52
% imsref loaded by jurgita.kaciuliene, 2014-02-07 11:44:00
% imsref loaded by jurgita.kaciuliene, 2014-02-07 13:13:45

\printhistory

\end{document}